\documentclass[11pt,english,reqno]{amsart}
\usepackage{bm}
\usepackage{color}
\usepackage{babel}
\usepackage{verbatim}
\usepackage{amsthm}
\usepackage{amstext}
\usepackage{amssymb}
\usepackage{amsbsy}
\usepackage{ifpdf}
\usepackage[nobysame,abbrev,alphabetic]{amsrefs}
\usepackage{enumerate}
\usepackage{amsmath}
\usepackage{amscd}
\usepackage{amsfonts}
\usepackage{graphicx}
\usepackage[all]{xy}
\usepackage{verbatim}
\usepackage{gensymb}
\usepackage{mathrsfs}
\usepackage[colorlinks]{hyperref}
\hypersetup{
linkcolor=blue,          
citecolor=red,        
}
\newtheorem{theorem}{Theorem}[section]
\newtheorem{claim}[theorem]{Claim}
\newtheorem{lemma}[theorem]{Lemma}
\newtheorem{corollary}[theorem]{Corollary}
\theoremstyle{definition}\newtheorem{definition}[theorem]{Definition}

\newtheorem{conjecture}[theorem]{Conjecture}

\theoremstyle{theorem}\newtheorem{proposition}[theorem]{Proposition}

\theoremstyle{definition}

\theoremstyle{definition}
\theoremstyle{definition}\newtheorem{remark}[theorem]{Remark}
\theoremstyle{definition}\newtheorem*{acknowledgments}{Acknowledgments}
\newcommand{\al}{\alpha}

\newcommand{\Ga}{\Gamma}
\newcommand{\del}{\delta}

\newcommand{\Lam}{\Lambda}
\newcommand{\eps}{\epsilon}

\newcommand{\ka}{\kappa}

\newcommand{\om}{\omega}

\newcommand{\vphi}{\varphi}

\newcommand{\cH}{\mathcal{H}}

\newcommand{\cM}{\mathcal{M}}

\newcommand{\cP}{\mathcal{P}}

\newcommand{\bA}{\mathbb{A}}

\newcommand{\bP}{\mathbb{P}}
\newcommand{\bR}{\mathbb{R}}
\newcommand{\bZ}{\mathbb{Z}}
\newcommand{\bQ}{\mathbb{Q}}

\newcommand{\bN}{\mathbb{N}}
\newcommand{\bH}{\mathbb{H}}

\newcommand{\SL}{\operatorname{SL}}
\newcommand{\SO}{\operatorname{SO}}

\newcommand{\PGL}{\operatorname{PGL}}

\newcommand{\GL}{\operatorname{GL}}
\newcommand{\PSL}{\operatorname{PSL}}
\newcommand{\defi}{\overset{\on{def}}{=}}

\newcommand\norm[1]{||#1||}

\newcommand\set[1]{\left\{#1\right\}}
\newcommand\pa[1]{\left(#1\right)}

\newcommand\av[1]{|#1|}
\newcommand\on[1]{\operatorname{#1}}

\newcommand\mb[1]{\mathbf{#1}}

\newcommand\smallmat[1]{\pa{\begin{smallmatrix}#1\end{smallmatrix}}}

\newcommand{\wstar}{\overset{\on{w}^*}{\lra}}

\newcommand{\lra}{\longrightarrow}

\newcommand{\onto}{\xymatrix{\ar@{>>}[r]&}}

\newcommand{\QQ}{\mathbb{Q}}
\newcommand{\PP}{\mathbb{P}}

\newcommand{\HH}{\mathbb{H}}
\newcommand{\ZZ}{\mathbb{Z}}
\newcommand{\NN}{\mathbb{N}}

\newcommand{\RR}{\mathbb{R}}

\newcommand{\limfi}[2]{{\displaystyle \lim_{#1\to#2}}}
\newcommand{\pp}{\mathcal{P}}

\newcommand{\dt}{\mathrm{dt}}

\newcommand{\dmu}{\mathrm{d\mu}}

\newcommand{\flr}[1]{\left\lfloor #1\right\rfloor }
\newcommand{\nuga}{\nu_{\mathrm{Gauss}}}
\newcommand{\inv}[1]{\left(\mathbb{Z}/#1\mathbb{Z}\right)^{\times}}
\newcommand{\len}{\mathrm{len}}
\newcommand{\nicefrac}[2]{#1/#2}

\begin{document}
%
\title[Equidistribution of divergent orbits]{Equidistribution of divergent orbits and continued fraction expansion
of rationals}

\author{Ofir David and Uri Shapira}

\address{Department of Mathematics, Technion, Haifa, Israel }

\email{ushapira@tx.technion.ac.il}

\address{Department of Mathematics, Hebrew University, Jerusalem, Israel }

\email{ofir.david@mail.huji.ac.il}

\thanks{}
\maketitle

\begin{abstract}
We establish an equidistribution
result for push-forwards of certain locally finite algebraic measures
in the adelic extension of the space of lattices in the plane. As
an application of our analysis we obtain new results regarding the
asymptotic normality of the continued fraction expansions of most
rationals with a high denominator as well as an estimate on the length
of their continued fraction expansions. 

By similar methods we also establish a complementary result to Zaremba's conjecture. Namely, we show 
that given a bound $M$, for any large $q$, the number of rationals $p/q\in[0,1]$ for which 
the coefficients of the continued fraction expansion of $p/q$ are bounded by $M$ is $o(q^{1-\eps})$
for some $\eps>0$ which depends on $M$.
\end{abstract}

\maketitle

\section{Introduction}
\subsection{Continued fraction expansion of rationals}

We begin by describing the main application of our results. Let $T:(0,1]\to[0,1]$
denote the Gauss map $T(s):=\left\{ s^{-1}\right\} :=s^{-1}-\lfloor s^{-1}\rfloor$.
Let $\nuga=((1+s)\ln2)^{-1}ds$ denote the Gauss-Kuzmin measure on
$[0,1]$. A number $s\in(0,1]$ is rational if and only if $T^{i}(s)=0$
for some $i$ (in which case $T^{i+1}(s)$ is not defined). In this
case we denote this $i$ by $\len(s)$ which is the length of the
(finite) continued fraction expansion of $s$ (hereafter abbreviated
c.f.e). We also set 
\[
\nu_{s}=\frac{1}{\len(s)}\sum_{i=0}^{\len\left(s\right)-1}\delta_{T^{i}\left(s\right)}.
\]
Throughout we abuse notation and denote $$(\bZ/q\bZ)^{\times}=\set{1\le p\le q:\gcd(p,q)=1}.$$
 \begin{theorem}
\label{thm:main_app} 
There exist sets $W_q\subseteq \inv{q}$ with $\displaystyle{\lim_{q\to \infty}} \frac{|W_q|}{\varphi(q)}=1$, such that for any choice  of $p_q \in W_q$ we have that
\begin{enumerate}
\item $\frac{\len(p_q/q)}{2\ln(q)}\to \frac{\ln(2)}{\zeta(2)}$ where $\zeta$ is the Riemann zeta function.
\item\label{p:2} $\nu_{p_q/q}\wstar \nuga$
\end{enumerate}
\end{theorem}
\begin{remark}
Let $\mb{w}$ be a finite word on $\bN$.
It is well known, and indeed follows 
from the ergodicity of $T$ with respect to $\nuga$, that for Lebesgue almost any $x$ the asymptotic
frequency  of appearances of $\mb{w}$ in the c.f.e of $x$ equals 
\begin{equation}\label{eq:1357}
\nuga(\mb{w})\defi \nuga\pa{\set{y\in[0,1]:\textrm{ the c.f.e of \ensuremath{y} starts with \ensuremath{\mb{w}}}}}.
\end{equation}
Let us denote by $\nu_{p/q}(\mb{w})$ the \emph{frequency}  of the 
word $\mb{w}$ in the c.f.e of $p/q$; that is, the number of appearances of $\mb{w}$ in the 
c.f.e of $p/q$ divided by $\len(p/q)$. 
Then, it is easy to see that since the endpoints of the interval given by the set in~\eqref{eq:1357} have zero $\nuga$ measure, 
then the weak* convergence in part \eqref{p:2} of Theorem~\ref{thm:main_app} implies that for any finite word $\mb{w}$ over $\bN$ we have that $\nu_{p_q/q}(\mb{w})\to \nuga(\mb{w})$.

\end{remark}

An obvious corollary of Theorem~\ref{thm:main_app} (together with the fact that $\len(p/q)\leq 2\log_2(q)$) is obtained by averaging
over $p\in\inv q$ as follows.
\begin{corollary}
\label{cor:app} 
\begin{enumerate}
\item \label{cor:app1} Let $\bar{\nu}_{q}=\vphi(q)^{-1}\sum_{p\in\inv q}\nu_{p/q}$.
Then $\bar{\nu}_{q}\wstar\nuga$. 
\item \label{cor:app2} Let $\overline{\len}(q)=\vphi(q)^{-1}\sum_{p\in\inv q}\len(p/q)$. 
Then $\frac{\overline{\len}(q)}{2\ln q}\to\frac{\ln2}{\zeta(2)}$. 
\end{enumerate}
\end{corollary}
This corollary was first obtain by Heilbronn 
\cite{heilbronn_average_1969}
who also computed an error term, which was later improved by Ustinov
\cite{ustinov_number_2009}. The upgrade from Corollary \ref{cor:app}
to Theorem \ref{thm:main_app} is almost automatic when the discussion
is lifted to the space of lattices as can be seen in \S\ref{subsec:upgrade}.
It seems not to be available when the discussion stays in the classical
realm of the Gauss map. Running over all $1\leq p\leq q$ and not
just $\left(p,q\right)=1$, Bykovskii \cite{bykovskii_estimate_2007}
showed that $\frac{1}{q}\sum_{1}^{q}\left(\len\left(\frac{p}{q}\right)-\frac{2\ln\left(2\right)}{\zeta\left(2\right)}\ln q \right)^{2}\ll\ln q $.

We note also that averaged versions of Theorem~\ref{thm:main_app} with an
extra average over $q$ were obtained by Dixon \cite{dixon_number_1970}
who showed that for any $\varepsilon>0$ there exists $c>0$ such
that 
\[
\#\set{ 
(p,q):
\begin{array}{ll}
{\tiny 1\leq p\leq q\leq x, }\\
{\tiny 
\left|\frac{\len\left(p/q\right)}{2\ln\left(2\right)}-\frac{\ln q }{\zeta(2)}\right|<\frac{1}{2}\left(\ln q \right)^{-\frac{1}{2}+\varepsilon}
}
\end{array}
} 
\leq x^{2}\exp\left(-c\ln^{\varepsilon/2}\left(x\right)\right),
\]
which was later improved by Hensley in \cite{hensley_number_1994}.
See also \cite{adler_construction_1981} and \cite{vandehey_new_2016}
for construction of normal numbers with respect to c.f.e using rational
numbers.
\subsection{Contrast to Zaremba's conjecture}

Recall that Zaremba's conjecture \cite{zaremba1972methode} asserts that there exists $M>0$ such that for all $q$
there exists $p\in \inv{q}$ such that all the coefficients in the c.f.e of $p/q$ are bounded by $M$. 
Theorem~\ref{thm:main_app} may be interpreted as saying that Zaremba is looking for a needle in a haystack. In fact, while Theorem~\ref{thm:main_app}
asserts that the set of $p/q$ which are good for Zaremba is of size $o(q)$, the following strengthening says that it is actually
$o(q^{1-\eps})$.
\begin{theorem}\label{thm:Zaremba}
For each $M$ there exists $\eps>0$ such that 
$$\#\set{p\in\inv{q} : \textrm{{\tiny the coefficients of the c.f.e of $p/q$ are bounded by $M$} }}  = o(q^{1-\eps}).$$
\end{theorem}
\subsection{\label{subsec:Divergent-geodesics}Divergent geodesics}

Let $G=\PGL_{2}(\bR)$, $\Ga=\PGL_{2}(\bZ)$ and $X_{2}=\Gamma\backslash G$.
The space $X_{2}$ is naturally identified with the space of homothety
classes of lattices in the plane where the coset $\Ga g$ corresponds
to the (homothety class of the) lattice $\bZ^{2}g$. We shall refer
to $\bZ^{2}$ as the standard lattice and denote its class in $X_{2}$
by $x_{0}$. We let $G$ and its subgroups act on $X_{2}$ from the
right and usually abuse notation and write elements of $G$ as matrices.
Consider the subgroups of $G$, 
\begin{equation}\label{eq:1518}
A=\set{a(t)=\left(\begin{smallmatrix}e^{-t/2} & 0\\
0 & e^{t/2}
\end{smallmatrix}\right):t\in\bR};\;U=\set{u_{s}=\left(\begin{smallmatrix}1 & s\\
0 & 1
\end{smallmatrix}\right):t\in\bR}
\end{equation}

Theorem~\ref{thm:main_app} is a consequence of a certain equidistribution
theorem regarding collections of divergent orbits of the diagonal
group which we now wish to discuss. It is not hard to see that if
$s=p/q$ is a rational in reduced form then the $A$-orbit $x_{0}u_{s}A$
is divergent; that is, the map $t\mapsto x_{0}u_{s}a(t)$ is a proper
embedding of $\bR$ in $X_{2}$. In fact, for $t<0$ this lattice
contains the vector $e^{t/2}(0,1)$ which is of length $e^{t/2}\to0$
as $t\to-\infty$ and for $t>0$ the lattice contains the vector $\left(q,-p\right)u_{s}a(t)=\left(qe^{-t/2},0\right)$
which is of length $\le1$ when $t\ge2\ln q$ and goes to zero as
$t\to\infty$. So the interesting life-span of the orbit $x_{0}u_{s}A$
is the interval $\set{x_{0}u_{s}a(t):t\in\left[0,2\ln q\right]}$.
We therefore define for $p\in\inv{q}$,
\begin{equation}
\delta_{x_0u_{p/q}}^{\left[0,2\ln q \right]}=\frac{1}{2\ln q }\int_{0}^{2\ln q }\delta_{x_{0}u_{p/q}a(t)}\dt\label{eq:mupq}
\end{equation}
(which means that for a bounded continuous function on $X_{2}$ we
have \\$\int_{X_{2}}fd\delta_{x_0u_{p/q}}^{\left[0,2\ln q \right]}:=\frac{1}{2\ln q}\int_{0}^{2\ln q}f(x_{0}u_{p/q}a(t))\dt$).
Finally, let $\mu_{Haar}$ denote the unique $G$-invariant probability
measure on $X_2$. The tight relation between the $A$-action on $X_2$ and continued
fractions is well understood. Indeed, we deduce Theorem~\ref{thm:main_app} from results in the space $X_2$
which we now describe.
\begin{theorem}
\label{thm:main} As $q\to\infty$ we have that 
\[
\frac{1}{\vphi(q)}\sum_{p\in\inv q}\delta_{x_0u_{p/q}}^{\left[0,2\ln q \right]}\wstar\mu_{Haar}.
\]
\end{theorem}

\begin{corollary}\label{cor:almost_all_converge}
There exist sets $W_q\subseteq \inv{q}$ with $\displaystyle{\lim_{q\to \infty}} \frac{|W_q|}{\varphi(q)}=1$, such that for any choice  of $p_q \in W_q$ we have that $\delta_{x_0 u_{p/q}}^{[0,2\ln(q)]} \wstar \mu_{Haar}$.
\end{corollary}

As mentioned before, although it seems stronger, Corollary \ref{cor:almost_all_converge} follows from Theorem~\ref{thm:main} using only the fact that $\mu_{Haar}$ is $A$-ergodic. See \S\ref{subsec:upgrade} for details.

We will prove Theorem~\ref{thm:main} as a consequence 
of the following more general equidistribution result. We say that a sequence of 
probability measures $\eta_n$ does not exhibit escape of mass if any weak* accumulation point of it is a probability measure.
\begin{theorem}\label{thm:mainug}
Let $\Lambda_q\subset (\bZ/q\bZ)^\times$ be subsets such that 
\begin{enumerate}[(i)]
\item\label{ass:01} $\lim \frac{\ln |\Lambda_q|}{\ln q} = 1$, 
\item\label{ass:02} the sequence of
measures $\frac{1}{|\Lambda_q|}\sum_{p\in \Lambda_q} \delta_{x_0u_{p/q}}^{[0,2\ln q]}$ does not exhibit escape of mass. 
\end{enumerate}
Then 
$\frac{1}{|\Lambda_q|}\sum_{p\in \Lambda_q} \delta_{x_0u_{p/q}}^{[0,2\ln q]}\wstar \mu_{Haar}$.
\end{theorem}
\begin{remark}
Note that in Theorem~\ref{thm:main} we have that $\av{\Lambda_q}=\varphi(q)$ is the Euler's totient function and it is well known
that $\lim \frac{\ln \varphi(q)}{\ln q} = 1$ which is condition \eqref{ass:01} above (indeed, this claim follows from the multiplicative nature of 
the totient function). Thus, in order to deduce Theorem~\ref{thm:main} from Theorem~\ref{thm:mainug} we only need
to show that there is no escape of mass.
\end{remark}
\subsection{\label{subsec:A-more-conceptual}A more conceptual viewpoint}

Let $X_{n}=\PGL_{n}(\bZ)\backslash\PGL_{n}(\bR)$ be identified with
the space of homothety classes of lattices in $\bR^{n}$ and let $A<\PGL_{n}(\bR)$
denote the connected component of the identity of the full diagonal group. 
It is well known (see \cite{tomanov_closed_2003})
that an orbit $xA$ is divergent (i.e. the map $a\mapsto xa$ from
$A$ to $X_{n}$ is proper), if and only if it contains a homothety
class of an integral lattice. It is not hard to show that in this
case there is a unique such integral lattice which minimizes the covolume.
We refer to the square of this covolume as the \textit{discriminant}
of the divergent
orbit. Let $\cH_{q}(n)$ be the finite collection of sublattices of $\bZ^{n}$
of covolume $q$ having the property that $\pi_{i}(\Lambda)=\bZ$
for $i=1,\dots n$, where $\pi_{i}$ is the projection onto the $i$'th
axis. We leave it as an exercise 
to show that the collection of divergent
orbits of discriminant $q^2$ is exactly $\set{xA:x\in\cH_{q}(n)}$.
By abuse of notation we also think of $\cH_{q}(n)$ as a subset of
$X_{n}$. In dimension 2 we have $\cH_{q}(2)=\set{\bZ^{2}\left(\begin{smallmatrix}1 & p\\
0 & q
\end{smallmatrix}\right):p\in\inv q}$. Note that the collection of orbits $\set{x_{0}u_{p/q}A:p\in\inv q}$
in $X_{2}$ is the same as $\set{xA:x\in\cH_{q}(2)}$.

In Theorem~\ref{thm:main} we truncated the divergent orbits $\set{xA:x\in\cH_{q}(2)}$,
since we wanted to use the weak{*} topology which is defined on the
space of \textit{finite} measures on $X_{2}$. It is conceptually
better to present a certain topology on the space of \textit{locally
finite} measures which will allow Theorem~\ref{thm:main} to be restated and
conveniently generalized to a convergence statement involving the
natural locally finite $A$-invariant measures supported on the collection
of divergent orbits $\set{xA:x\in\cH_{q}(2)}$. To this end, let us
denote by $\mu_{xA}$ the measure on $X_{2}$ obtained by pushing
a fixed choice of Haar measure on $A$ via the map $a\mapsto xa$
(where $xA$ is divergent and hence the map is proper so that the
pushed measure is indeed locally finite). In dimension 2 we identify
$A\simeq\bR$ by $t\mapsto a(t)$ and choose the standard Lebesgue
measure coming from this identification.

Let $Z$ be a locally compact second countable Hausdorff space and
let $\cM(Z)$ denote the space of locally finite positive Borel measures
on $Z$ and let $\bP\cM(Z)$ denote the space of homothety classes
of such (non-zero) measures. For $\mu\in\cM(Z)$ we let $\left[\mu\right]$
denote its class. It is straightforward to define a topology on $\bP\cM(Z)$
such that the following are equivalent for $\left[\mu_{n}\right],\left[\mu\right]\in\bP\cM(Z)$
(see~\cite{ShapiraZheng}),
\begin{enumerate}
\item $\lim\left[\mu_{n}\right]=\left[\mu\right]$. 
\item\label{2ndconv} There exist constants $c_{n}$ such that for any compact set $K\subset Z$,
$c_{n}\mu_{n}|_{K}\wstar\mu|_{K}$ (which means that for every $f\in C_{c}(Z)$,\\
$c_{n}\int fd\mu_{n}\to\int fd\mu$). 
\item For every $f,g\in C_{c}(Z)$ for which $\int gd\mu\ne0$, $\limfi n{\infty}\frac{\int fd\mu_{n}}{\int gd\mu_{n}}\to\frac{\int fd\mu}{\int gd\mu}$
(and in particular, $\int gd\mu_{n}\ne0$ for all large enough $n$). 
\end{enumerate}
It is straightforward to see that if $c_{n},c_{n}'$ are sequences
of scalars such that $c_{n}\mu_{n}$ and $c_{n}'\mu_{n}$ both converge
to $\mu$  in the sense of  \eqref{2ndconv}, then $c_{n}/c_{n}'\to1$.

We propose the following. 
\begin{conjecture}
\label{conj:hd} For any dimension $n$, as $q\to\infty$, the homothety
class of the locally finite measure $\sum_{x\in\cH_{q}(n)}\mu_{xA}$
converges in the above topology to the homothety class of the $\PGL_{n}(\bR)$-invariant
measure on $X_{n}$.
\end{conjecture}
\begin{theorem}
\label{thm:main2} Conjecture~\ref{conj:hd} holds for $n=2$.
\end{theorem}
We will see in  Lemma~\ref{lem:to_locally_finite} that Theorem~\ref{thm:main2} 
follows from (and is in fact equivalent to) Theorem~\ref{thm:main}.

\subsection{Adelic orbits}

We now concentrate on the 2-dimensional case. Yet another conceptual
view point that we wish to present and which puts the statement of
Theorem~\ref{thm:main2} in a natural perspective is as follows.
Let $\bA$ denote the ring of adeles over $\bQ$ and 
consider the space $X_{\bA}=\Gamma_{\bA}\backslash G_{\bA}$ (where
$G_{\bA}=\PGL_{2}(\bA)$ and $\Ga_{\bA}=\PGL_{2}(\bQ)$). Let
$A_{\bA}<G_{\bA}$ denote the subgroup of diagonal matrices.
Note that
the orbit $\tilde{x}_{0}A_{\bA}$ is a closed orbit (where $\tilde{x}_{0}$
denotes the identity coset $\Ga_{\bA}$). In particular, fixing once
and for all a Haar measure on $A_{\bA}$ we obtain a Haar measure
on the quotient $\on{stab}_{A_{\bA}}(\tilde{x}_{0})\backslash A_{\bA}$
and by pushing the latter into $X_{\bA}$ via the proper embedding
induced by the map $a\mapsto \tilde{x}_{0}a$ we obtain an $A_{\bA}$-invariant
locally finite measure $\mu_{\tilde{x}_{0}A_{\bA}}$ supported on
the closed orbit $\tilde{x}_{0}A_{\bA}$. Theorem~\ref{thm:main2}
(and hence Theorem~\ref{thm:main})
is implied (and in fact equivalent as will be seen by the proof) to
the following. 
\begin{theorem}
\label{thm:main3} For any sequences $g_{i}\in G_{\bA}$
such that (i) the real component of $g_{i}$ is trivial, (ii) the
projection of $g_{i}$ to $G_{\bA}/ A_{\bA}$ is unbounded,
the sequence of homothety classes of the locally finite measures $(g_{i})_{*}\mu_{\tilde{x}_{0}A_{\bA}}$
converges in the topology introduced above to the homothety class
of the $G_{\bA}$-invariant measure on $X_{\bA}$. 
\end{theorem}
In fact we propose the following. 
\begin{conjecture}
\label{conj:general2dim} In the statement of Theorem~\ref{thm:main3}
one can omit requirement (i) from the sequence $g_{i}$. 
\end{conjecture}
The main result in~\cite{oh2014limits} can be interpreted as saying that
if $g_{i}\in\PGL_{2}(\bR)$ is unbounded modulo the diagonal group
$A$, then the homothety class of $(g_{i})_{*}\mu_{x_{0}A}$ converges
in the topology introduced above to the homothety class of $\mu_{Haar}$.
It seems plausible (although not immediate as far as we can see) that
a proof of Conjecture~\ref{conj:general2dim} might be obtained by
combining the techniques of~\cite{oh2014limits} and ours. 

\subsection{Structure of the paper and outline of the proofs}

In \S\ref{sec:max_entropy_for_partial_orbits} we prove Theorem~\ref{thm:mainug}. 
We show that any weak* accumulation point of the sequence of measures appearing 
in the statement
(which
is automatically $A$-invariant) has the same entropy with respect
to say, $a(1)$, as the measure $\mu_{Haar}$. Since $\mu_{Haar}$
is the unique measure with maximal entropy this establishes that $\mu_{Haar}$
is the only possible weak{*} accumulation point of the above sequence
and finishes the proof. 
We then deduce Theorem~\ref{thm:main} by verifying that the two conditions for 
applying Theorem~\ref{thm:mainug} hold for $\Lam_q = \inv{q}$. Here the non-trivial part
is to show that in this case there is no
escape of mass.

In \S\ref{sec:adeles} we prove that Theorems~\ref{thm:main}, \ref{thm:main2},
\ref{thm:main3} are equivalent. In \S\ref{sec:Application-to-CFE} we
review the relation between the $A$ action on $X_2$ and the Gauss
map and isolate the necessary technical statements which will allow
us to deduce Theorem~\ref{thm:main_app} from Theorem~\ref{thm:main}. We end \S\ref{sec:Application-to-CFE}
by proving Theorem~\ref{thm:Zaremba} the proof of which follows along similar lines as the proof of Theorem~\ref{thm:main_app}.

\begin{acknowledgments}
The authors would like to thank Manfred Einsiedler for valuable discussions and acknowledge the support of ISF grant 357/13.
\end{acknowledgments}
\section{\label{sec:max_entropy_for_partial_orbits}Proof of the main theorem}
In this section we prove Theorem~\ref{thm:mainug} and deduce Theorems~\ref{thm:main}. 
We start with some notation and definitions and then, 
in \S\ref{ssec:nad} make a 
minor reduction to replace the measures that appear in the statement of Theorem~\ref{thm:mainug} with  a 
discrete version of themselves which is better suited for the entropy argument. In \S\ref{subsec:Maximal_entropy} we 
state the main tool we use in the proof - uniqueness of measure with maximal entropy - and establish maximal entropy
of the appropriate weak* limits which finishes the proof of Theorem~\ref{thm:mainug}. In 
\S\ref{subsec:No_escape_of_mass} we verify that the measures
appearing in the statement of Theorem~\ref{thm:main} satisfy the conditions in Theorem~\ref{thm:mainug} 
and by that conclude the 
proof of Theorem~\ref{thm:main}. Finally, in \S\ref{subsec:upgrade} we use the ergodicity of the Haar measure 
in order to upgrade the averaged result from Theorem~\ref{thm:main} to Corollary~\ref{cor:almost_all_converge}.

In this section we set $G=\SL_{2}\left(\RR\right),\;\Gamma=\SL_{2}\left(\ZZ\right)$
and are interested in equidistribution in the space $X=X_{2}=\Gamma\backslash G\cong\PGL_{2}\left(\ZZ\right)\backslash\PGL_{2}\left(\RR\right)$.
The group $G$ then acts naturally on $X$ and on the space of
functions on $X$. We denote the positive diagonal and upper
unipotent subgroups of $\SL_{2}\left(\RR\right)$ by $A,U$ respectively as in~\eqref{eq:1518}.

As mentioned in \S\ref{subsec:Divergent-geodesics}, we will work with
measures on partial $A$-orbit defined as follows.
\begin{definition}

\begin{enumerate}[(i)]
\item For a finite set $\Lambda\subseteq X$ we write $\delta_\Lambda=\frac{1}{|\Lambda|} \sum_{x\in \Lambda} \delta_x$. 
	We will sometimes write $\delta_{p/q}$ instead of $\delta_{x_0u_{p/q}}$, and given a set $\Lambda_q\subseteq \inv{q}$, 
	we will identify it with the set $\set{x_0 u_{p/q} : p\in \Lambda_q}\subseteq X$, and simply write $\delta_{\Lambda_q}$.
\item Given a measure $\mu$, a segment $[a,b]\subseteq \mathbb{R}$ and an integer $k\in \ZZ$, we define the averages $\mu^{[a,b]}=\frac{1}{b-a}\int_a^b a(-t)\mu \dt$ and $\mu^k=\frac{1}{k} \sum_0^{k-1} a(-j)\mu $. Note that with these definitions
$\del_x^k = \frac{1}{k}\sum_0^{k-1} \del_{xa(j)}$ and similarly, $\del_x^{[a,b]} = \frac{1}{b-a}\int_a^b \del_{xa(t)} dt$.
\end{enumerate}
\end{definition}
\subsection{A reduction}\label{ssec:nad}

The following statement  is very similar to that of Theorem~\ref{thm:mainug}. The only difference is that 
the continuous interval $[0,2\ln q]$ is replaced by the discrete first half of it $\bZ\cap [0,\ln q]$.
\begin{theorem}\label{thm:mainug2}
Let $\Lam_q\subset \inv{q}$ be subsets such that 
\begin{enumerate}[(i)]
\item\label{ass:1} $\lim\frac{\ln \av{\Lam_q}}{\ln q} = 1$,
\item \label{ass:2} The sequence of measures
$\delta_{\Lam_q}^{\flr{\ln q}}$ does not exhibit escape of mass (that is, any weak* limit of it is a probability measure). 
\end{enumerate}
Then $\delta_{\Lam_q}^{\flr{\ln q}}\wstar \mu_{Haar}$.
\end{theorem}
For entropy considerations it will be more convenient to work with powers of a single transformation rather than 
with the continuous group $A$. As will be seen shortly, replacing $[0,2\ln q]$ by its first half will also be more convenient. Thus 
our plan is to establish Theorem~\ref{thm:mainug2} but first we deduce Theorem~\ref{thm:mainug} from it.
\begin{proof}[Proof of Theorem~\ref{thm:mainug} given Theorem~\ref{thm:mainug2}]
Assume $\Lam_q$ satisfies assumptions (i) and (ii) of Theorem~\ref{thm:mainug}. Let $\tau:X\to X$ be the automorphism 
taking a lattice to its dual and recall that if $x = \Ga g$ then $\tau(x) = \Ga (g^{-1})^{tr}$, where $tr$ means the transpose, 
and hence $\tau(xa(t))=\tau(x)a(-t)$ for all $t\in \RR$. Let us denote also $p\mapsto p'$
the map from $\inv{q}\to\inv{q}$ for which $p p' = -1$ modulo $q$. We claim that 
\begin{equation}\label{eq:symmetry}
 \delta_{\Lam_q}^{[\ln q ,2\ln q]} =\tau_* \delta_{\Lam_q'}^{[0, \ln q]}.
\end{equation}
To show~\eqref{eq:symmetry} we first observe the following: Fix $p\in\inv{q}$ and let $q'\in\bZ$ 
be such that $(-p)p' + q q' = 1$. We then have

\begin{align*}
 x_0u_{p/q} a(2\ln q)  
& =  
\Gamma \smallmat{
1 & p/q\\
0 & 1}\smallmat{q^{-1}&0 \\ 0 &q} =
\Ga\smallmat{q^{-1}&p\\0&q} \\
&=
\Gamma\smallmat{q & -p\\-p' & q'}\smallmat{q^{-1}&p\\0&q} = \Ga\smallmat{1&0\\ -p'/q&1}
=  \tau(x_0 u_{p'/q}).
\end{align*}
It now follows that for all $t$, $ x_0u_{p/q} a(2\ln q-t)=\tau(x_0 u_{p'/q}a(t))$, and hence \eqref{eq:symmetry} follows.
We conclude from \eqref{eq:symmetry} that
\begin{equation}\label{eq:sim2}
\delta_{\Lam_q}^{[0,2\ln q]} = \frac{1}{2} \delta_{\Lam_q}^{[0,\ln q]} + \frac{1}{2}\tau_* \delta_{\Lam_q'}^{[0,\ln q]}. 
\end{equation}
Since $\delta_{\Lam_q}^{[0,2\ln q]}$ does not exhibit escape of mass, the same is true for 
the sequence $\delta_{\Lam_q}^{[0,\ln q]}$ (as well as $\delta_{\Lam_q'}^{[0,\ln q]}$). Since 
\begin{equation}\label{eq:convcomb}
\delta_{\Lam_q}^{[0,\ln q]}=\frac{\flr{\ln q}}{\ln q} \delta_{\Lam_q}^{[0,\flr{\ln q}]} + (1-\frac{\flr{\ln q}}{\ln q})\delta_{\Lam_q}^{[\flr{\ln q}, \ln q]},
\end{equation} and $\frac{\flr{\ln q}}{\ln q}\to 1$, we conclude that 
the sequence $\delta_{\Lam_q}^{[0,\flr{\ln q}]}$ does not exhibit escape of mass. Finally, since 
\begin{equation}\label{eq:integral}
\delta_{\Lam_q}^{[0,\flr{\ln q}]} = \int_0^1 a(-t)_* \delta_{\Lam_q}^{\flr{\ln q}} dt,
\end{equation} 
we conclude that $\delta_{\Lam_q}^{\flr{\ln q}}$
does not exhibit escape of mass. We therefore obtain $\Lam_q$ satisfy conditions (i) and (ii) from Theorem~\ref{thm:mainug2} and since we assume the validity of this theorem at this point, we conclude that $\delta_{\Lam_q}^{\flr{\ln q}}\wstar \mu_{Haar}$. Since $\mu_{Haar}$ is $a(t)$-invariant, equation~\eqref{eq:integral} implies that $\delta_{\Lam_q}^{[0,\flr{\ln q}]}\wstar \mu_{Haar}$. In turn, by \eqref{eq:convcomb} we get that $\delta_{\Lam_q}^{[0,\ln q]}\wstar \mu_{Haar}$. 

A similar application of Theorem~\ref{thm:mainug2} for $\Lam_q'$ results in the conclusion that $\delta_{\Lam_q'}^{[0,\ln q]}\wstar \mu_{Haar}$ and since $\mu_{Haar}$ is $\tau$-invariant, we obtain from~\eqref{eq:sim2} that $\delta_{\Lam_q}^{[0,2\ln q]}\wstar\mu_{Haar}$ as claimed. 
\end{proof}

\subsection{\label{subsec:Maximal_entropy}Maximal entropy}
We briefly recall the notion of entropy mainly to set the notation. 
The reader is referred to any standard textbook on the subject
for a more thorough account. See e.g.\ \cite{EinsiedlerWardEntropy, walters2000introduction}.
Recall that given a measurable space $(Y,\mathcal{B})$, a finite measurable partition
$\pp$ of $Y$ and a probability measure $\mu$ on $Y$ we define
the entropy of $\mu$ with respect to $\pp$ to be
\[
H_{\mu}\left(\pp\right)=-\sum_{P_{i}\in\pp}\mu\left(P_{i}\right)\ln\left(\mu\left(P_{i}\right)\right).
\]
We refer to the sets composing the partition $\cP$ as the \textit{atoms} of $\cP$. 
Given a $\mu$-preserving transformation $T:Y\to Y$, we define 
\begin{align*}
\forall k<\ell\in\bZ,\; \pp_k^\ell & =\bigvee_{i=k}^{\ell-1}T^{-i}\pp\\
h_{\mu}\left(T,\pp\right) & =\limfi n{\infty}\frac{1}{n}H_{\mu}\left(\pp_0^{n}\right)=\liminf_{n\geq1}\frac{1}{n}H_{\mu}\left(\pp_0^{n}\right)\\
h_{\mu}\left(T\right) & =\sup_{\left|\pp\right|<\infty}h_{\mu}\left(T,\pp\right)
\end{align*}
The following characterization of $\mu_{Haar}$ in terms of maximal entropy is the main tool
we use in the proof of Theorem~\ref{thm:mainug2}, where the map $T:X\to X$ is defined by 
$$T\left(x\right)=xa\left(1\right)=x\smallmat{e^{-t/2} & 0\\
0 & e^{t/2}}.$$
\begin{theorem}
[see \cite{ELPisa,einsiedler_distribution_2012}]\label{thm:maximalentropy} 
Let 
$\mu$ be a $T$-invariant probability measure on $X$.
Then $h_{\mu}\left(T\right)\leq h_{\mu_{Haar}}\left(T\right)=1$, and
there is an equality if and only if $\mu=\mu_{Haar}$.
\end{theorem}
In what follows all partitions of $X$ are implicitly assumed to be finite and measurable.
Suppose that $\delta_{\Lambda_q}^{\flr{\ln q }}\wstar\mu,\; \Lambda_q\subseteq \inv{q}$
for some sequence $q\to\infty$ and let $\pp$ be any partition of
$X$ such that the boundaries
of the atoms of $\pp$ have zero $\mu$-measure. This condition implies
that $H_{\mu}\left(\pp_0^m\right)=\limfi q{\infty}H_{\delta_{\Lambda_q}^{\flr{\ln q }}}\left(\pp_0^m\right)$.
Our goal in the end is to show that the entropy $h_{\mu}\left(T,\pp\right)$
is big for a well chosen partition $\pp$, or equivalently that $\frac{1}{m}H_{\mu}\left(\pp_0^m\right)$
is big when $m\to\infty$ which is translated to a suitable condition
on the entropy of $\delta_{\Lambda_q}^{\flr{\ln q }}$. 

Recall that for a finite set $\Lambda\subseteq\Gamma\backslash G$, the measure $\delta_{\Lambda}^{k}$
is the average of the measures $\delta_{x}^{k},\;x\in\Lambda$,
and each of these measures is an average along the $T$-orbit.
Switching the orders of these averages we get that
\begin{align*}
\delta_{\Lambda}^{k} & =\frac{1}{\left|\Lambda\right|}\sum_{x\in\Lambda}\frac{1}{k}\sum_{i=0}^{k-1}\delta_{x a\left(i\right)}
=\frac{1}{k}\sum_{i=0}^{k-1}T^{i}\left(\frac{1}{\left|\Lambda\right|}\sum_{x\in\Lambda}\delta_{x}\right)=\frac{1}{k}\sum_{i=0}^{k-1}T^{i}\left(\delta_{\Lambda}\right).
\end{align*}
The concavity of the entropy function implies that $\delta_{\Lambda}^{k}$
has large entropy if most of the entropies of $T^{i}\left(\delta_{\Lambda}\right)$
are large, and these are all pushforwards of the same measure
$\delta_{\Lambda}$. With this idea in mind we have the following
result the proof of which is inspired by the proof of the variational principle in~\cite{EinsiedlerWardEntropy}.

\begin{lemma}
\label{lem:entropy_lower_bound}Let $Y$ be any measurable space,
let $S:Y\to Y$ be some measurable function, $\pp$ a partition of $Y$ and $\mu$ a probability
measure on $Y$. We denote by $\mu^{k}=\frac{1}{k}\sum_{i=0}^{k-1}S^{i}\mu$.
Then
\begin{enumerate}
\item If $\mu=\sum_{1}^{k}a_{i}\mu_{i}$ is a convex combination of 
probability measures $\mu_{i}$, then
$
H_{\mu}\left(\pp\right)\geq\sum_{1}^{k}a_{i}H_{\mu_{i}}\left(\pp\right)
$.
\item For every $n,m\in \NN$, we have that 
\[
\frac{1}{m}H_{\mu^{n}}\left(\pp_0^m\right)\geq\frac{1}{n}H_{\mu}\left(\pp_0^{n}\right)-\frac{m}{n}\ln\left|\pp\right|
\]
\end{enumerate}
\end{lemma}

\begin{proof}

\begin{enumerate}
\item Since the function $\alpha:x\mapsto-x\ln\left(x\right)$ is concave
in $\left[0,1\right]$, we obtain that
\begin{align*}
H_{\mu}\left(\pp\right)&=\sum_{P\in\pp}\alpha\left(\mu\left(P\right)\right) 
=\sum_{P\in\pp}\alpha\left(\sum_{1}^{k}a_{i}\mu_{i}\left(P\right)\right)\\
&\geq\sum_{1}^{k}a_{i}\sum_{P\in\pp}\alpha\left(\mu_{i}\left(P\right)\right)
=\sum_{1}^{k}a_{i}H_{\mu_{i}}\left(\pp\right).
\end{align*}

\item Write $n=km+r\leq m\left(k+1\right)$ where $0\leq r<m$. Using subadditivity we get that for $0\leq u\leq m-1$ we have

\begin{eqnarray*}
H_{\mu}\left(\pp_0^{n}\right) & \leq & H_{\mu}\left(\pp_0^{km+r}\right)\\
 & \leq & \sum_{i=0}^{u-1}H_{\mu}\left(S^{-i}\pp\right)+\sum_{v=0}^{k-1}H_{\mu}(S^{-\left(vm+u\right)}\pp_{m})+\sum_{i=dm+u}^{dm+m-1}H_{\mu}(S^{-i}\pp)\\
 & \leq & m\log\left|\pp\right|+\sum_{v=0}^{k-1}H_{S^{vm+u}\mu}(\pp_0^m).
\end{eqnarray*}
Summing over $0\leq u\leq m-1$ we get that 

\begin{eqnarray*}
mH_{\mu}\left(\pp_0^{n}\right) - m^{2}\ln\left|\pp\right| & \leq & \sum_{u=0}^{m-1}\sum_{v=0}^{k-1}H_{\left(S^{vm+u}\mu\right)}(\pp_0^{m})
 \leq  \sum_{j=0}^{km-1}H_{\left(S^{j}\mu\right)}(\pp_0^{m})\\
 & \leq & \sum_{j=0}^{n-1}H_{\left(S^{j}\mu\right)}(\pp_0^{m})
  \leq  n H_{\mu^{n}}\left(\pp_0^{m}\right),
\end{eqnarray*}
where in the last step we used part (1). It then follows that \\$\frac{1}{m}H_{\mu^{n}}\left(\pp_0^{m}\right)\geq\frac{1}{n}H_{\mu}\left(\pp_0^{n}\right)-\frac{m}{n}\ln\left|\pp\right|.$
\end{enumerate}
\end{proof}

\begin{corollary}
\label{cor:m_to_q}Suppose that $\Lambda_q\subset (\bZ/q\bZ)^\times$ and  
$\delta_{\Lambda_q}^{\flr{\ln q }}\wstar \mu$ along some sequence of $q$'s for a measure $\mu$ on $X$.
Then, if $\pp$ is a partition whose atoms have
boundary of zero $\mu$-measure, then $h_{\mu}(T,\pp)\geq{\displaystyle \limsup_{q\to\infty}}\frac{1}{\flr{\ln q }}H_{\left(\delta_{\Lambda_q}\right)}(\pp_{0}^{\flr{\ln q }})$.
\end{corollary}
\begin{proof}
Follows from Lemma~\ref{lem:entropy_lower_bound} and since $\frac{m}{\flr{\ln q }}\ln\left|\pp\right|\to0$
as $q\to\infty$.
\end{proof}
By the corollary above, we are left with the problem of showing that
${\displaystyle \limsup_{q\to\infty}}\frac{1}{\flr{\ln q }}
H_{\left(\delta_{\Lambda_q}\right)}(\pp_0^{\flr{\ln q }})$
is big. Suppose that we can show that for every $S\in\pp_0^{\flr{\ln q }}$, $|S\cap \Lambda_q|\le r$
or in other words $\delta_{\Lambda_q}(S)\le \frac{r}{|\Lambda_q|}$.
This would imply
\begin{align}\label{eq:1424}
\frac{1}{\flr{\ln q}}H_{\delta_{\Lambda_q}}(\pp_0^{\flr{\ln q}}) 
& = \frac{1}{\flr{\ln q}}\sum_{S\in \pp_0^{\flr{\ln q}}} \delta_{\Lambda_q}(S)\ln \frac{1}{\delta_{\Lambda_q}(S)} \\
\nonumber & \ge \frac{1}{\flr{\ln q}}\sum_{S\in \pp_0^{\flr{\ln q}}} \delta_{\Lambda_q}(S)\ln \frac{|\Lambda_q|}{r} 
= \frac{\ln |\Lambda_q|}{\flr{\ln q}} - \frac{\ln r}{\flr{\ln q}}.
\end{align}
If $|\Lambda_q|$ is big enough and $r$ is small enough; i.e.\ $ \frac{\ln |\Lambda_q|}{\flr{\ln q}} - \frac{\ln r}{\flr{\ln q}}\to 1$, then 
we get the lower bound that we wish to establish. We will follow this line of argument with a certain complication that arises. The bound $r$ will basically come from the fact that the diameter of $S$ is small and the points of $\Lam_q$ are well separated,
 but in fact, one cannot control uniformly 
the diameter of the atoms of $\pp_0^{\flr{\ln q}}$. Lemma~\ref{lem:bowen_control} below shows that one can find a partition
for which one can do so for most
atoms.
Before stating Lemma~\ref{lem:bowen_control}
we introduce some terminology.

Recall that $X$ is naturally identified with the space of unimodular lattices in the plane. 
For a lattice $x\in X$ we define the height of $x$ to be 
$$\on{ht}(x)=\max\set{ \norm{v}^{-1}:0\ne v\in x}$$
and set $X^{\le M}=\set{ x\in X:\on{ht}(x) \le  M}$
which is compact (similarly we define $X^{<M},X^{\geq M},X^{>M}$).
Under this notation $X=\bigcup_{1}^{\infty}X^{\leq M}$
is $\sigma$-compact. 
\begin{definition}\label{def:balls}
For $H\leq \SL_{2}\left(\RR\right)$, define $B_{r}^{H}=\left\{ I+W\in H:\norm W_{\infty}<r\right\} $.
In particular for $U^{+}, U^{-}A\leq \SL_2(\bR)$ we have $B_{r}^{U^{+}}=\left\{ I+tE_{1,2}:\left|t\right|<r\right\} $
and $B_{r}^{U^{-}A}=\left\{ I+W\in SL_{2}\left(\RR\right):W_{1,2}=0,\;\left|W_{i,j}\right|<r\right\} $. We also write\\
$B_{\eta,N}=B_{\eta e^-N}^{U^+} B_\eta^{U^-A}$, $B_\eta:=B_{\eta,0}$.
\end{definition}
\begin{definition}\label{def:good partitions}
A (finite measurable) partition $\pp$ of $X$ is
called an $\left(M,\eta\right)$ partition if $\pp=\left\{ P_{0},P_{1},...,P_{n}\right\} $
where $P_{0}=X^{>M}$ and $P_{i}\subseteq x_{i}B_{\eta}$, $x_{i}\in X$
for $1\leq i\leq n$. If $\mu$ is a probability measure on $X$,
then $\pp$ is called an $\left(M,\eta,\mu\right)$ partition if in
addition $\mu\left(\partial P_{i}\right)=0$ for all $i$
\end{definition}
\begin{remark}\label{rem:good partitions}
Given a measure $\mu$ one can construct $(M,\eta,\mu)$-partitions for arbitrary large $M$ and arbitrary small
$\eta$ in abundance. To see this we note that $\mu(\partial X^{>M}) = 0$ outside a countable set of $M$'s and after 
defining $P_0=X^{>M}$ one defines the $P_i$'s by a disjointification procedure starting with a finite cover of the compact set $X^{\le M}$ by balls of arbitrarily small radius having $\mu$-null boundary. The point here being is that for a given center $x$, 
outside a countable set or radii $\mu(\partial xB_r)=0$.
\end{remark}
Lemma~\ref{lem:bowen_control} is a slight adaptation of Lemma 4.5 from \cite{einsiedler_distribution_2012}. 
For convenience, we added
the full proof in Appendix~\ref{app:Working-with-balls} (see also Remark~\ref{rem:correction}).

\begin{lemma}[Existence of good partitions~\cite{einsiedler_distribution_2012}]\label{lem:bowen_control} 
For any $M>1$ there exists some $0<\eta_{0}\left(M\right)$ such that for any $0<\eta\leq\eta_{0}\left(M\right)$ and an
$\left(M,\frac{1}{10}\eta\right)$ partition $\pp$ of $X$ the following holds:
For any $\kappa\in\left(0,1\right)$ and any $N>0$, there exists some $X'\subseteq X^{\leq M}$
such that 
\begin{enumerate}
\item $X'$ is a union of $S_{1},...,S_{l}\in \pp_0^N$;
\item Each such $S_{j}$ is contained in a union of at most $C^{\kappa N}$
many balls of the form $zB_{\eta,N}$
with $z\in S_{j}$ for some absolute constant $C$.
\item  $\mu(X')\geq1-\mu\left(X^{>M}\right)-\mu^{N}\left(X^{>M}\right)\kappa^{-1}$
for any probability measure $\mu$ on $X$ (where $\mu^N=\frac{1}{N}\sum_{n=0}^{N-1}T^n_*\mu$).
\end{enumerate}
\end{lemma}

Lemma~\ref{lem:bowen_control} 
gives us the tool to produce partitions whose entropies could be controlled in the proof of Theorem~\ref{thm:mainug2}.
The last bit of information we need before turning to the proof of Theorem~\ref{thm:mainug2} is the following separation lemma.
\begin{lemma}[Good Separation]\label{lem:separation}
Let $p_{1},p_{2}\in (\ZZ/q\ZZ)^\times$.
If $\Gamma u_{p_{1}/q},\Gamma u_{p_{2}/q}\in 
zB_{\eta,\flr{\ln q }},$ for some  $\eta<\frac{1}{100}$
then $p_{1}=p_{2}$.
\end{lemma}
\begin{proof}
Given the assumption, there exist some $b_1,b_2\in B_{\eta,\flr{\ln(q)}}$ such that $\Gamma u_{p_i/q}=zb_i$,
and hence $u_{-p_1/q}\gamma u_{p_2/q}=b_1^{-1}b_2$ for some $\gamma=\smallmat{
a & b\\
c & d
}\in \SL_{2}\left(\ZZ\right)$. Applying Lemma~\ref{lem:BALLS},
this is contained in $B_{10\eta, \flr{\ln(q)}}$. On the other hand, this expression equals to 
\begin{equation}\label{eq:1444}
\smallmat{
1 & -\frac{p_{1}}{q}\\
0 & 1
}
\smallmat{
a & b\\
c & d
}
\smallmat{
1 & \frac{p_{2}}{q}\\
0 & 1
}=
\smallmat{
a-\frac{p_{1}}{q}c & b-\frac{p_{1}}{q}d+\frac{p_{2}}{q}\left(a-\frac{p_{1}}{q}c\right)\\
c & d+\frac{p_{2}}{q}c
}
\end{equation}

We conclude that $c$, the bottom left coordinate,
is at most $10 \eta<1$ in absolute value, so that $c=0$. It then follows similarly that $a=d=1$.
We are then left with the top right coordinate which is $b+\frac{p_{2}-p_{1}}{q}$
which need to be at most $\left(1+10 \eta\right)10\eta e^{-\flr{\ln{q}}}<\frac{1}{q}$
in absolute value, so we must have that $p_{1}=p_{2}$ and we are
done.
\end{proof}
Finally, after collecting all the above information we are in a position to prove 
Theorem~\ref{thm:mainug2} (and by that complete also the proof of Theorem~\ref{thm:mainug}).
\begin{proof}[Proof of Theorem~\ref{thm:mainug2}]
It is enough to show that $\mu_{Haar}$ is the only accumulation point of $\delta_{\Lam_q}^{\flr{\ln q}}$.
Let $\mu$ be such an accumulation point, which is necessarily $T$-invariant and by assumption \eqref{ass:2} is a probability measure, and restrict attention to a sequence of $q$'s for which 
$\delta_{\Lambda_q}^{\flr{\ln q }}\wstar \mu$. 
We shall show that $h_{\mu}(T) =1$ and therefore by Theorem~\ref{thm:maximalentropy} conclude that $\mu = \mu_{Haar}$ as desired.

By Corollary \ref{cor:m_to_q}, for a partition $\pp$ whose atoms have boundary of zero $\mu$-measure we have that
\begin{equation}\label{eq:1600}
h_\mu (T,\pp)\geq \limsup_q \frac{1}{\flr{\ln q}} H_{\delta_{\Lambda_q}}(\pp_0^{\flr{\ln q }}).
\end{equation}

Let $\cP$ be an $(M,\eta,\mu)$-partition (see Definition~\ref{def:good partitions} and Remark~\ref{rem:good partitions}). 
Fix $\ka>0$ and $N=\flr{\ln(q)}$ and let $X'$ be as in Lemma~\ref{lem:bowen_control}. 
If $P\in \pp^{\flr{\ln q}}$ is such that $P\subseteq X'$, then Lemma~\ref{lem:bowen_control} implies 
that it can be covered by $C^{\kappa \flr{\ln q}}$ sets which by Lemma~\ref{lem:separation} contain 
at most one element from $\Lambda_q$ each. 
This translates to the bound $\delta_{\Lambda_q}(P)\leq \frac{1}{|\Lambda_q|}C^{\kappa\flr{\ln(q)}}$ and therefore,
\begin{align}\label{eq:1601}
\frac{1}{\flr{\ln q}}H_{(\delta_{\Lambda_q})}(\pp_0^{\flr{\ln q }})&\geq 
-\frac{1}{\flr{\ln q}}\sum_{P\subseteq X'}\delta_{\Lambda_q}(P)\ln(\delta_{\Lambda_q}(P))\\
\nonumber&\geq  -\frac{1}{\flr{\ln q}}\sum_{P\subseteq X'}\delta_{\Lambda_q}(P)\ln(\frac{1}{|\Lambda_q|}C^{\kappa\flr{\ln q}}) \\
\nonumber&=\frac{1}{\flr{\ln q}}\delta_{\Lambda_q}(X') \pa{\ln|\Lambda_q|-\kappa\flr{\ln q }\ln(C)}\\
\nonumber & \geq \pa{ 1-\delta_{\Lam_q}^{\flr{\ln q}}(X^{\ge M})\ka^{-1}} \pa{\frac{\ln|\Lambda_q|}{\ln q}- \kappa\ln C}.
\end{align}
Given $\eps>0$, using assumptions \eqref{ass:1} and \eqref{ass:2}, namely $\lim \frac{ \ln\av{\Lam_q} } {\ln{q}}=1$ and $\displaystyle{\lim_{M\to \infty}\lim_{q\to \infty}} \delta_{\Lam_q}^{\flr{\ln q}}(X^{\ge M})=0$, we see that  
we can choose $M$ to be big enough and $\ka$ to be small enough so that for all large enough $q$ the expression
on the right in~\eqref{eq:1601}
is $\ge (1-\eps)(1-\eps)$. We conclude from~\eqref{eq:1600} that 
$h_\mu(T) = \sup_{\cP} h_\mu(T,\cP)\ge 1$ which concludes the proof.

\end{proof}

\subsection{\label{subsec:No_escape_of_mass}No escape of mass}
 Our goal in this section is to prove Theorem~\ref{thm:main} by showing that the sets $\Lam_q = \inv{q}$ satisfy 
 the conditions \eqref{ass:01} and \eqref{ass:02} of Theorem~\ref{thm:mainug}.
Throughout this section we set $\Lam_q = \inv{q}$ and $\mu_q = \delta_{\Lam_q}$. 

We begin with verifying that condition~\eqref{ass:01} holds which is the content of the following lemma.  
\begin{lemma}\label{lem:ass1holds}
As $q\to\infty$, $\frac{\ln \vphi(q)}{\ln q}\to 1$.
\end{lemma}
\begin{proof}
Fix $q$ and let $p_i, i=1,\dots \om(q)$ be its prime divisors. Since 
\begin{equation}\label{eq:phiq}
\vphi(q) = q \prod_{i=1}^{\om(q)} (1-p_i^{-1})
\end{equation}
we have that 
$$\ln \vphi(q) =\ln q +  \sum_{i=1}^{\om(q)} \ln(1-p^{-1}) \ge \ln q + \sum_{i=1}^{\om(q)}\ln(1/2) = \ln q -\om(q)\ln 2.$$
We conclude that $$ 1-\frac{\om(q)}{\ln{q}}\ln 2\le \frac{\ln{\vphi(q)}}{\ln q}\le 1$$
and since it was shown  by Robin in \cite{robin_estimation_1983} that
$\omega\left(q\right)=O\left(\frac{\ln q }{\ln\ln q }\right)$ we conclude 
that $\frac{\ln \vphi(q)}{\ln q}\to 1$ as desired.
\end{proof}
Showing that condition~\eqref{ass:2} is satisfied for $\Lam_q$ is the content of Lemma~\ref{lem:no_escape} 
below. We proceed towards 
its proof by establishing several lemmas.
The following simple lemma basically says that $\Lam_q$ is equidistributed on the circle.

\begin{lemma}\label{lem:equi on the circle}
Let $q$ be some integer and $0\leq\alpha\leq1$. Then 
$$\av{ \# \set{ 1\leq \ell \leq\alpha q: \ell \in\inv{q} } -\al\vphi(q) }\leq 2^{\om(q)}$$
where $\om(q)$ is the number of distinct prime factors of $q$. 


\end{lemma}
\begin{proof}
Let $p$ be a prime that divides $q$ and set $U_{p}=\set{ 1\leq \ell \leq\alpha q :p|  \ell }$.
We want to find $\flr{\al q} - \av{ \cup_{p_{i}}U_{p_{i}} }$
where $p_{i}$ are the distinct primes that divide $q$.

Using inclusion exclusion we get that
\begin{eqnarray*}
\flr{\al q}-\av{ \cup_{p}U_{p} } & = & \flr{\al q }-\sum_{i}\av{ U_{p_{i}} } + \sum_{i<j}\av{ U_{p_{i}} \cap U_{p_{j}}} 
+ \cdots + (-1)^{\om(q)}\av{ \cap_{i}U_{p_{i}}} \\
 & = & \flr{\al q}-\sum_{i}\flr{\frac{\al q}{p_{i}}}+\sum_{i<j}\flr{\frac{\al q}{p_{i}p_{j}}}+\cdots+ 
 (-1)^{\om(q)}\flr{\frac{\al q}{\prod_{i}p_{i}}}.
\end{eqnarray*}
On the other hand, using~\eqref{eq:phiq}, we have that 
\[
\al\vphi(q) =\al q \prod_{1}^{\om(n)} (1-\frac{1}{p_{i}})=\al q-\sum_{i}\frac{\al q}{p_{i}}+\sum_{i<j}\frac{\al q}{p_{i}p_{j}}
+ \cdots + (-1)^{\om(q)}\frac{\al q}{\prod_{i}p_{i}}
\]
so that 
\[
\av{ \al\vphi(q) - (\flr{\al q} - \av{\cup_{p}U_{p}})} \leq \sum_{k=0}^{\om(q)}\binom{\om(q)}{k}=2^{\om(q)}.
\]
\end{proof}
The following lemma is the heart of the argument yielding the validity of condition~\eqref{ass:2} and in fact establishes 
a much stronger non-escape of mass than the one we need, namely it shows that there is no escape of mass 
for any sequence of measures of the form $a(-t_q)_*\mu_q$ where $q\to\infty$ and $t_q$ is allowed to vary almost 
without constraint in the interval $[0,\ln q]$; namely it is allowed to vary in $[0,\ln q - 2\om(q)]$. 
\begin{lemma}[No escape of mass]\label{lem:nem}
 Fix some $q\in\NN,\;M>1$ and $0\leq t\leq\ln q -2\omega\left(q\right)$.
Then
\[
\left|\left\{ p\in (\bZ/q\bZ)^\times:\Gamma u_{p/q}a\left(t\right)\in X^{\geq M}\right\} \right|\leq\frac{4}{M^{2}}\varphi\left(q\right).
\]
Equivalently, $a(-t)_*\mu_q(X^{\ge M})< \frac{4}{M^2}.$
\end{lemma}
\begin{proof}
We say that $p$ is bad if $\Gamma u_{p/q}a\left(t\right)\in X_2^{\geq M}$. 
Thus, $p$ is bad
if and only if there exists a vector
\begin{eqnarray*}
v_{p}\left(m,n,t\right) & = & \left(m,n\right)\left(\begin{smallmatrix}1 & \frac{p}{q}\\
0 & 1
\end{smallmatrix}\right)\left(\begin{smallmatrix}e^{-t/2} & 0\\
0 & e^{t/2}
\end{smallmatrix}\right)=\left(me^{-t/2},\left(n+m\frac{p}{q}\right)e^{t/2}\right)\\
\end{eqnarray*}
such that
\begin{eqnarray*}
\norm{v_{p}\left(m,n,t\right)}^{2} & = & m^{2}e^{-t}+\left(n+m\frac{p}{q}\right)^{2}e^{t}\leq\frac{1}{M^{2}},\quad\left(m,n\right)\neq\left(0,0\right).
\end{eqnarray*}
In particular, this implies that $\left(n+m\frac{p}{q}\right)^{2}e^{t}\leq\frac{1}{M^{2}}$ and $m\le \frac{e^{t/2}}{M}$.
We may also assume that $m\geq0$ and in fact that $m\neq0$,
since otherwise $\left(n+m\frac{p}{q}\right)^{2}e^{t}=n^{2}e^{t}\geq n^{2}\geq1>\frac{1}{M^{2}}$
using the assumption that $t\geq0$.  Let us say that $p$ is bad for $m\in [1, \frac{e^{t/2}}{M}]$ if there exists $n$ such that 
$\left|n+m\frac{p}{q}\right|\leq\frac{1}{e^{t/2}M}$.
We will bound the number of bad $p$'s by bounding the number of 
bad $p$'s for each $m\in [1, \frac{e^{t/2}}{M}]$.

 Given such $m$ and bad $p$ we can find
$n$ such that $\left|n+m\frac{p}{q}\right|\leq\frac{1}{e^{t/2}M}$
or equivalently $\left|nq+mp\right|\leq\frac{q}{e^{t/2}M}$. Letting
$d_{m}=\gcd\left(q,m\right)$ and writing $q=\tilde{q}d_{m},\;m=\tilde{m}d_{m}$,
we get that 
\begin{equation}\label{eq:1131}
\left|\tilde{q}n+\tilde{m}p\right|\leq\frac{\tilde{q}}{e^{t/2}M}.
\end{equation}
We will bound the number of $p$'s solving~(\ref{eq:1131}) by considering its meaning in the ring $\bZ/\tilde{q}\bZ$. 
Note that $m\leq\frac{e^{t/2}}{M}\leq\frac{\sqrt{q}}{M}$
so that $\tilde{q}=\frac{q}{\left(q,m\right)}\geq M\sqrt{q}>1$. This allows us to 
consider the group $(\bZ/\tilde{q}\bZ)^\times$ and the natural surjective 
homomorphism $\pi:(\bZ/q\bZ)^\times\to (\bZ/\tilde{q}\bZ)^\times$.
Furthermore, since $\tilde{m},p\in(\bZ/\tilde{q}\bZ)^\times$ the meaning of the 
inequality~(\ref{eq:1131}) may be interpreted in $(\bZ/\tilde{q}\bZ)^\times$.
Namely, if we let 
$\Omega=\left\{ \left[a\right]\in\left(\nicefrac{\ZZ}{\tilde{q}\ZZ}\right)^{\times}:\left|a\right|\leq\frac{\tilde{q}}{e^{t/2}M}\right\} $,
then the bad $p$'s for $m$ are exactly $\pi^{-1}\left(\tilde{m}^{-1}\Omega \right)$,
hence there are at most $\left|\Omega \right|\cdot\left|\ker\left(\pi\right)\right|$
such $p$. 
Since $\pi$ is surjective we obtain that
$\left|\ker\left(\pi\right)\right|=\frac{\varphi\left(q\right)}{\varphi\left(\tilde{q}\right)}$
and by Lemma~\ref{lem:equi on the circle} we get that $\left|\Omega \right|\leq2\left(\frac{1}{e^{t/2}M}\varphi\left(\tilde{q}\right)+2^{\omega\left(\tilde{q}\right)}\right)$.

We claim that $2^{\omega\left(\tilde{q}\right)}\leq\frac{1}{e^{t/2}M}\varphi\left(\tilde{q}\right)$.
Assuming this claim, the total number of bad $p$'s (for a fixed $m$)
is at most $\left|\Omega\right|\cdot\left|\ker\left(\pi\right)\right|\leq\frac{4}{e^{t/2}M}\varphi\left(q\right)$.
Since there are $\flr{\frac{e^{t/2}}{M}}$ such $m$, a union bound
shows that the number of bad $p$ is at most $\frac{4}{e^{t/2}M}\varphi\left(q\right)\frac{e^{t/2}}{M}=\frac{4}{M^{2}}\varphi\left(q\right)$.
Thus, to complete the proof we need only to show that $\frac{2^{\omega\left(\tilde{q}\right)}}{\varphi\left(\tilde{q}\right)}\leq\frac{1}{e^{t/2}M}$.
From~\eqref{eq:phiq} it follows that for any $k$,
$\vphi(k)\ge k(\frac{1}{2})^{\omega(k)}$ and so we deduce that 
\begin{eqnarray*}
\frac{2^{\omega\left(\frac{q}{d_{m}}\right)}}{\varphi\left(\frac{q}{d_{m}}\right)} & \leq & \frac{2^{\omega\left(\frac{q}{d_{m}}\right)}}{(\frac{1}{2})^{\omega\left(\frac{q}{d_{m}}\right)}\frac{q}{d_{m}}}=\frac{4^{\omega\left(\frac{q}{d_{m}}\right)}}{q}d_{m}\leq 
\frac{e^{2\omega\left(\frac{q}{d_{m}}\right)}}{q}\frac{e^{t/2}}{M}\\
& \leq &
\exp\left(t+2\omega\left(\frac{q}{d_{m}}\right)-\ln q \right)\frac{1}{e^{t/2}M} \le \frac{1}{e^{t/2}M}
\end{eqnarray*}
where the last inequality follows from the fact that $\omega\left(q\right)\geq\omega\left(\frac{q}{d_{m}}\right)$,
and our assumption that  $ t\leq\ln q -2\omega\left(q\right)$ so that $t +2\omega\left(\frac{q}{m}\right)-\ln q \le 0$.
\end{proof}
We now conclude the validity of condition~\eqref{ass:2} by averaging the result of Lemma~\ref{lem:nem} over $t\in[0,\ln q]$.
\begin{lemma}
\label{lem:no_escape}
For any $q>1$ and any $M>1$ 
we have 
$$\mu_{q}^{\flr{\ln q }}\left(X^{<M}\right)\geq1-\left(\frac{4}{M^{2}}+O\left(\frac{1}{\ln\ln q }\right)\right).$$
\end{lemma}
\begin{proof}
Using the previous lemma we get that
\begin{align*}
\mu_{q}^{\flr{\ln q}}\left(X^{\geq M}\right) & =  \frac{1}{\flr{\ln q}}
\sum_{k=0}^{\flr{\ln q}-1} a(-k)_*\mu_q (X^{\geq M})\\
 & \leq \frac{1}{\flr{\ln q }}
 \sum_{k=0}^{\flr{\ln q }-2\omega\left(q\right)-1} a(-k)_*\mu_q(X^{\geq M}) +\frac{2\omega\left(q\right)}{ \flr{\ln q}}
 \le \frac{4}{M^{2}}+\frac{2\omega\left(q\right)}{\flr{\ln q }}.
\end{align*}
Finally, it was shown by Robin in \cite{robin_estimation_1983} that
$\omega\left(q\right)=O\left(\frac{\ln q }{\ln\ln q }\right)$
, thus completing the proof.
\end{proof}
\begin{proof}[Proof of Theorem~\ref{thm:main}]
By Lemmas~\ref{lem:ass1holds}, \ref{lem:no_escape} the two conditions \eqref{ass:1} and \eqref{ass:2} 
of Theorem~\ref{thm:mainug} are satisfied for $\Lam_q=\inv{q}$ yielding the result.
\end{proof}

\subsection{\label{subsec:upgrade}Upgrading the main result}

Theorem~\ref{thm:main} tells us that the averages $\delta_{\Lambda_q}^{[0,2\ln(q)]}$ where $\Lambda_q= \inv{q}$ converge to the Haar measure. The ergodicity of the Haar measure allows us to automatically upgrade this result to subsets of $\inv{q}$ of positive proportion.

\begin{theorem} \label{thm:extreme_ergodic}
Let $1\geq\alpha>0$ and choose $W_{q}\subseteq \inv{q} $
such that $\left|W_{q}\right|\geq\alpha\varphi\left(q\right)$ for
every $q$. Then $\delta_{W_{q}}^{\left[0,2\ln q \right]}\overset{w^{*}}{\longrightarrow}\mu_{Haar}$.
\end{theorem}
\begin{proof}
Let $\mu$ be an accumulation point of $\delta_{W_{q_{i}}}^{\left[0,2\ln\left(q_{i}\right)\right]}$
for some subsequence $q_{i}$ (which is necessarily $A$-invariant). Going down to a subsequence, we may
assume that $\frac{\left|W_{q_{i}}\right|}{\varphi\left(q_{i}\right)}\to\alpha_{0}\geq\alpha>0$ and $\delta_{\Lambda_{q_i} \backslash W_{q_i}}^{[0,2\ln(q_i)}\to \mu'$ converge.
We now have that
\[
\mu_{q_i^{[0,2\ln(q)]}}=\frac{\left|W_{q_{i}}\right|}{\varphi\left(q_{i}\right)}\cdot\delta_{W_{q_{i}}}^{\left[0,2\ln\left(q_{i}\right)\right]}+\frac{\left|\Lambda_{q_{i}}\backslash W_{q_{i}}\right|}{\varphi\left(q_{i}\right)}\cdot\delta_{\Lambda_{q_{i}}\backslash W_{q_{i}}}^{\left[0,2\ln\left(q_{i}\right)\right]},
\]
and taking the limit we get that 
\[
\mu_{Haar}=\alpha_{0}\mu+\left(1-\alpha_{0}\right)\mu'.
\]
This is a convex combination of $A$-invariant probability measures with positive $\alpha_0$. The ergodicity of $\mu_{Haar}$ implies that it is extreme point in the set of $A$-invariant probability measures, hence we conclude that $\mu=\mu_{Haar}$. As this is true for any convergenct subsequence of $\delta_{W_{q}}^{\left[0,2\ln q \right]}$,
we conclude that it must converge to the Haar measure.
\end{proof}

Once we have the convergence result for any positive proportion sets, we also automatically get a second upgrade
and show that almost all choices of sequence $\delta_{p_i/q_i}^{[0,\ln(q_i)]}$ converge.

\begin{proof} [Proof of Corollary~\ref{cor:almost_all_converge}]
Let $\mathcal{F}=\{f_1,f_2,...\}$ be a countable dense family of continuous functions in $C_c(X_2)$. For each $n,q\in \bN$ define 
$$W_{q,n}=\{p\in \inv{q} : \max_{1\leq i \leq n} |(\delta_{p/q}^{[0,2\ln(q)]}-\mu_{Haar})(f_i)|<\frac{1}{n}\}.$$
We claim that $\displaystyle{\lim_{q\to \infty}} \frac{|W_{q,n}|}{\varphi(q)}=1$ for any fixed $n$. Otherwise, we can find some $1\leq i\leq n$ ,$\epsilon \in \{\pm 1\}$ and $\alpha>0$ such that the set
$$V_q=\{p\in \inv{q} : \epsilon (\delta_{p/q}^{[0,2\ln(q)]}-\mu_{Haar})(f_i)\geq\frac{1}{n}\}$$
satisfies $\frac{|V_{q_j}|}{\varphi(q_j)} \geq \alpha$  for some subsequence $q_j$.
By Theorem~\ref{thm:extreme_ergodic} we obtain that $\delta_{V_{q_j}}^{[0,2\ln(q_j)]}\wstar \mu_{Haar}$, while  $\epsilon(\delta_{V_{q_j}}^{[0,2\ln(q_j)]}-\mu_{Haar})(f_i)\geq\frac{1}{n}$ for all $j$ - contradiction (note that $i,n$ are fixed).

We conclude that for any $n$ there exists $q_n$ such that for any $q\ge q_n$, 
$\frac{|W_{q,n}|}{\varphi(q)}\ge 1-1/n$. Without loss of generality we may assume that $q_n$ is strictly monotone. We then define 
for any $q$, $n_q=\max\set{n : q\ge q_n}$. It then follows that $W_q:=W_{q,n_q}$ satisfies that 
that $n_q \to \infty$ and $\frac{W_q}{\varphi(q)}\to 1$ as $q\to \infty$. We are left to show that $\delta_{p_q/q}^{[0,2\ln(q)]} \wstar \mu_{Haar}$ for any choice of seqeunce $p_q\in W_q$. By the definition of $W_q$, for any fixed $i$ we have that $\delta_{p_q/q}^{[0,2\ln(q)]}(f_i) \to \mu_{Haar}(f_i)$, 
and since $\mathcal{F}$ is dense in $C_c(X_2)$, this claim holds for any $f\in C_c(X_2)$, or in other words $\delta_{p_q/q}^{[0,2\ln(q)]} \wstar \mu_{Haar}$.
\end{proof}

\section{\label{sec:adeles}Equidistribution over the adeles}
In this section we prove Theorem~\ref{thm:main3} which is an enhancement of Theorem~\ref{thm:main}. We establish this equidistribution statement in the adelic 
space $X_\bA:=\PGL_2(\bQ)\backslash \PGL_2(\bA)$ which we refer to as the adelic extension of $X_\bR:= X_2$.



We shall start in Subsection~\ref{subsec:local_finite} with some general results about locally finite measures and their push 
forwards. In particular we shall prove a ``compactness" criterion that roughly states that
if the push forward of a sequence of locally finite measures converges to a probability measure, 
then it has a subsequence that converges to a probability measure.

In Subsection~\ref{subsec:lifting_haar} we prove that the Haar measure on $X_\bR$ has a unique lift
to an $A_\bR$-invariant measure in $X_\bA$. Finally, in Subsection~\ref{subsec:lifting_orbit} we show
that the $A_\bR$-invariant measures $\mu_q = \sum_{p\in \inv{q}} \mu_{u_{p/q}A}$ (which are the orbit measure counterparts of the measures appearing in Theorem~\ref{thm:main}), are projections of measures on $X_\bA$ which are obtained as push-forwards of 
a single orbit measure in the adelic space, and show that this sequence
converges to the Haar measure on $X_\bA$. This will lead us to the proof of Theorem~\ref{thm:main3}.

\subsection{\label{subsec:local_finite}Locally finite measures}

In this section all the spaces are locally compact second countable Hausdorff spaces.
A measure on a space $Z$ is called \textit{locally finite} if every point in $Z$ has
a neighborhood with finite measure. Since $Z$ is locally compact, this is equivalent to saying 
that every compact set has a finite measure. We denote the space of locally finite measures
by $\cM(Z)$ and the space of homothety classes of such (non-zero) measure by $\bP\cM(Z)$.
Recall that we say that $[\nu_i]\to [\nu]$ for nonzero measures $\nu_i,\nu \in \cM(Z)$ if
there exist scalars $c_i>0$ such that $c_i \mu_i\mid _K \wstar \nu\mid _K$ 
for any compact subset $K\subseteq Z$. 

Given two spaces $X,Y$ and a continuos proper
map $\pi:X\to Y$, we obtain a map $\cM(X) \to \cM(Y)$ and its homothethy counterpart $\bP\cM(X)\to \bP\cM(Y)$,
both of which we shall denote by $\pi_*$. We will be interested in lifting convergent sequences from $\bP\cM(Y)$ to $\bP \cM(X)$.
 The next theorem is a type of compactness criterion
which assures us that we can lift at least a convergent subsequence. Moreover, if we can show that
the limit measure on $Y$ has a unique preimage measure on $X$, then the convergence in $Y$ will imply a convergence in $X$.

\begin{theorem}
\label{thm:locally_finite_converge}
Let $\pi:X\to Y$ be a continuous proper map and let $\nu_i \in \cM(X)$ and $\tilde{\nu_i}=\pi_*(\nu_i) \in \cM(Y)$.
If $[\tilde{\nu}_i]\to [\tilde \nu]$ for some probability measure $\tilde{\nu}$ on $Y$, then $[\nu_{i_k}]\to [\nu]$ for 
some subsequence $i_k$ and a probability measure $\nu$ on $X$ such that $\pi_*(\nu)=\tilde{\nu}$.

If in addition $[\nu_i]$ are in a closed subset $\Omega \subseteq \bP\cM(X)$ which contains a unique preimage $[\nu]$ 
of $[\tilde{\nu}]$, then $[\nu_i] \to [\nu]$.
\end{theorem}
\begin{proof}
Multiplying $\nu_i$ by suitable scalars, we may assume that $\tilde{\nu}_i\mid _K \wstar \tilde{\nu}\mid_K$ 
 for every compact $K \subseteq Y$. 
It then follows that $\nu_{i,K}:=\nu_i \mid_ {\pi^{-1}\left(K\right)}$ are finite with uniform bound, since  
$\nu_{i,K}(X)=\tilde{\nu}_i(K)\to \tilde{\nu}(K)\leq 1$. Choose a sequence of compact sets $K_{j}\nearrow Y$ such that
 any compact $K\subseteq Y$ is contained in some $K_j$ for some $j$, which implies the same conditions on $\pi^{-1}(K_j)$. 
Applying the Banach-Alaoglu theorem, we can find a subsequence $i_k$ such that $\nu_{i_k,K_j}$ converges as $k\to \infty$ for every $j$, which implies that $\nu_{i_k}\to \nu$ for some $\nu\in \cM(X)$. Clearly, we must have that $\pi_*(\nu)=\tilde{\nu}$, and $\nu$ must be a probability measure.

The second claim now follows from the first. Indeed, suppose that $\nu_i \in  \Omega \subseteq \bP\cM(X)$ which is closed and $\nu$ is the unique preimage of $\tilde{\nu}$ in $\Omega$. If the sequence $\nu_i$ doesn't converge to $\nu$, then it there is an open neighborhood $V$ of $\nu$ and a subsequence $\nu_{i_k}\notin V$. By the first claim this sequence has
a convergent subsequence, and since $\Omega$ is closed, it must converge to $\nu\in V$ - contradiction. Thus $\nu_i$ must
converge to $\nu$.
\end{proof}

\subsection{\label{subsec:lifting_haar} Lifts of the Haar measure}

For the rest of this section we fix the following notations. For a
set $S\subseteq\mathbb{P}$, where $\mathbb{P}$ is the set of primes
in $\NN$, we write 
\begin{align*}
 G_{S} & :=\PGL_{2}(\RR)\times{\textstyle \prod_{p\in S}'}\PGL_{2}(\QQ_{p})\quad;\quad H_{S}:=\PGL_{2}\left(\RR\right)\times{\textstyle \prod_{p\in S}}\PGL_{2}(\ZZ_{p})\\
 & \ZZ\left[S^{-1}\right]:=\ZZ\left[\frac{1}{p}:p\in S\right]\quad;\quad\Gamma_{S}:=\PGL_{2}(\ZZ\left[S^{-1}\right]).
\end{align*}
where $\prod_{p}'$ denotes the restricted product with respect to
$\PGL_{2}\left(\ZZ_{p}\right)$ (which is the standard product if $S$
is finite). Note that $H_{S}\leq G_{S}$ is a subgroup in a natural
way and $\Gamma_{S}$ is embedded as a lattice in $G_{S}$ via the diagonal map
$\gamma\mapsto\left(\gamma,\gamma,...\right)$, and we shall denote $X_S:=\Gamma_S \backslash G_S$.
In case that $S=\mathbb{P}$
or $S=\emptyset$, we will sometimes use the subscript $\bA$ (resp. $\RR$)
instead, and we remark that $\ZZ\left[\PP^{-1}\right]:=\QQ$ (resp.
$\ZZ\left[\emptyset^{-1}\right]:=\ZZ$). We denote by $\mu_{S,Haar}$ the Haar probability measure on $X_S$.

We will denote by $A_S$ the full diagonal subgroup in $G_S$. Note that $A$ is still reserved to the diagonal group with positive entries, namely the matrices $\{\left(\begin{smallmatrix}
e^{-t} & 0\\
0 & 1
\end{smallmatrix}\right)\}$ considered as a subgroup of $\PGL_2(\bR)$ while $A_\bR = \{\left(\begin{smallmatrix}
\pm e^{-t} & 0\\
0 & 1
\end{smallmatrix}\right)\}$.

Fixing $S\subseteq\mathbb{P}$, it is not hard to show that $H_{S}$
acts transitively on $X_S$ by using the fact
that $\QQ_{p}=\ZZ_{p}+\ZZ\left[\frac{1}{p}\right]$, thus leading
to the identification $X_S\cong \PGL_{2}\left(\ZZ\right)\backslash H_{S}$.
This induces the natural projections
\[
\pi_{S}^{S'}:X_{S'} \cong \PGL_{2}\left(\ZZ\right)\backslash H_{S'}\to \PGL_{2}\left(\ZZ\right)\backslash H_{S}\cong X_S \quad\forall S\subseteq S'\subseteq\PP.
\]

For any $S$, we have a $\PGL_2(\bR)$-right action on $X_{S}$ (and the induced $A_\bR$-action),
which commutes with the projections above. Moreover, these projections are easily seen to be proper 
since the only noncompact part of $H_S$ is $\PGL_2(\bR)$.
Thus, we can apply the results from the previous subsection. 

We start by showing that $\mu_{\bA,Haar}$ is the unique $A_\bR$-invariant lift of $\mu_{\bR,Haar}$. 
We shall prove this claim in two step -  first by lifting to $X_S$ with $S$ finite 
by using the maximal entropy method, and then for $X_\bA$
which follows from the structure of the restricted product.

In the following, we consider the actions by $T=\left(\begin{smallmatrix}e^{-1/2} & 0\\
0 & e^{1/2}
\end{smallmatrix}\right)$ and $U=\left\{ \left(\begin{smallmatrix}1 & s\\
0 & 1
\end{smallmatrix}\right):s\in\RR\right\} $  
on the spaces $X_S$ via their images in $\PGL_2(\bR)$.

Before considering $A_\bR$-invariant measures, we show that $\PGL_2(\bR)$-invariant measure on $G_S$ are always the Haar measure, by using the fact that $\Gamma_S$ and $\PGL_2(\bR)$ generate $G_S$.

\begin{lemma}\label{lem:projecting_measures}
Let $G_{1},G_{2}$ be unimodular locally compact second countable
Hausdorff groups and $\Gamma\leq G=G_{1}\times G_{2}$ a subgroup such
that $\overline{\left\langle G_{1},\Gamma\right\rangle }=G$. Then
a left $\Gamma$ and right $G_1$-invariant locally compact measure
$\mu$ on $G$ is the (right and left) Haar measure.
\end{lemma}
\begin{proof}
Consider the natural product map $C_{c}\left(G_{1}\right)\otimes C_{c}\left(G_{2}\right)\to C_{c}\left(G_{1}\times G_{2}\right)$
defined by $\left(f_{1}\otimes f_{2}\right)\left(g_{1},g_{2}\right)=f_{1}\left(g_{1}\right)f_{2}\left(g_{2}\right)$.
Using the Stone Weierstrass Theorem, we obtain that it has a dense
image (in the sup norm), hence it is enough to show that $\mu\left(R_{g}\left(\psi_{1}\otimes\psi_{2}\right)\right)=\mu\left(\psi_{1}\otimes\psi_{2}\right)$
for any $\psi_{i}\in C_{c}\left(G_{i}\right),\;i=1,2$ where $R_{g}$
(and later on $L_{g}$) is the right multiplication by $g$ (resp.
left).

If $g=g_{1}\in G_{1}$, then $\mu\left(R_{g_{1}}\left(\psi_{1}\otimes\psi_{2}\right)\right)=\mu\left(\left(R_{g_{1}}\psi_{1}\right)\otimes\psi_{2}\right)$
so by the right $G_{1}$-invariance of $\mu$ we learn that $\psi_{1}\mapsto\mu\left(\psi_{1}\otimes\psi_{2}\right)$
is right $G_{1}$-invariant. The unimodularity of $G_{1}$ implies that this map and therefore $\mu$ are left $G_1$-invariant. The set $Stab_{G}\left(\mu\right)=\left\{ g\in G\;\mid\;\mu\circ L_g = \mu\right\} $
is closed in $G$ and contains $\left\langle G_{1},\Gamma\right\rangle $,
so $\mu$ is left $G$-invariant. Finally, since $G$
is unimodular we conclude that $\mu$ is right $G$-invariant as well,
i.e. it is a Haar measure.
\end{proof}
\begin{lemma}[Unique ergodicity]\label{lem:unique_erg}
Let $S\subseteq\PP$ be finite and let $\mu_{S}$ be a $\PGL_2(\bR)$-invariant probability measure
on $X_S$. Then $\mu_S$ must be the Haar measure.
\end{lemma}
\begin{proof}
We will show that $\PGL_2(\bR)$ invariance together with the quotient by $\Gamma_S$ from the left in $\Gamma_S \backslash G_S$, implies that $\mu_S$ must be $G_S$ invariant.

Let $\tilde{\mu}_{S}$ be the lift of $\mu_{S}$ to $G_{S}$, i.e.
for sets $F$ inside the fundamenal domain we set $\tilde{\mu}_{S}\left(F\right)=\mu_{S}\left(\Gamma_{S}F\right)$,
and extend this to a left $\Gamma_S$-invariant measure on $G_{S}$. 
The measure $\tilde{\mu}_S$ is left $\Gamma_S$ and right $\PGL_2(\bR)$-invariant measure and using the weak approximation of $\ZZ\left[S^{-1}\right]$ in $\prod_{p\in S}\QQ_{p}$ we get that $\overline{<\Gamma_S,\PGL_2(\bR)>}=G_S$. Applying Lemma~\ref{lem:projecting_measures}, we obtain that it is the Haar measure on $G_S$, hence $\mu_S$ is right $G_S$-invariant which completes the proof.
\end{proof}
Next, we would like to show that $A_\bR$-invariance implies $\PGL_2(\bR)$-invariance.

\begin{theorem}[see Theorems 7.6 and 7.9 in \cite{ELPisa}]
\label{thm:max_entropy_U_invariant}Fix some finite set $S\subseteq\PP$
and let $\lambda$ be a $T$-invariant probability
measure on $X_S$. Then $h_{\lambda}\left(T\right)\leq1$
with equality if and only if $\lambda$ is $U$-invariant. Similarly,
$h_{\lambda}\left(T^{-1}\right)\leq1$ with equality if and only if
$\lambda$ is $U^{tr}$-invariant (where $U^{tr}$ is the transpose
of $U$).
\end{theorem}

\begin{theorem}
\label{thm:Lifting_to_finite_is_Haar}Let $S\subseteq\PP$ be finite
and let $\mu_{S}$ be an $A_\bR$-invariant probability measure on $X_S$, such that
\[
\left(X_S,\mu_{S},T\right)\overset{\pi_{\RR}^{S}}{\longrightarrow}(X_{\bR},\mu_{\bR,Haar},T)
\]
is a factor map. Then
$\mu_{S}=\mu_{S,Haar}$.
\end{theorem}
\begin{proof}
Since the entropy only decreases in a factor and the Haar measure
 is $U$-invariant, an application of Theorem~\ref{thm:max_entropy_U_invariant}
shows that $1\geq h_{\mu_{S}}\left(T\right)\geq h_{\mu_{\bR,Haar}}\left(T\right)=1$.
It follows that $h_{\mu_{S}}(T)=1$, and hence $\mu_{S}$
is also $U$ invariant. Repeating the process with $T^{-1}$,
we get that $\mu_{S}$ is $\left\langle U,U^{tr}, A_\bR \right\rangle= \PGL_2(\bR)$ invariant. The theorem now follows from Lemma~\ref{lem:unique_erg}.
\end{proof}
\begin{theorem}\label{thm:lift_Haar_adeles}Let $\mu_{\bA}$ be an $A_\bR$-invariant
probability measure on $X_{\bA}$ such that
$(X_{\bA},\mu_{\bA},T)\overset{\pi_{\RR}^{\bA}}{\to}(X_{\bR},\mu_{\bR,Haar},T)$
is a factor. Then
$\mu_{\bA}=\mu_{\bA,Haar}$.
\end{theorem}
\begin{proof}
For each finite $S\subseteq \bP$ we can pull back the functions in $C_c(X_S)$ to $C_c(X_\bA)$ and the union of
these sets over $S$ spans a dense subset of $C_c(X_\bA)$. Hence, it is enough to prove that for any such set $S$, $f\in C_c(X_S)$ and $g\in G_\bA$ we have that $\mu_\bA(g(f\circ \pi^\bA_S))=\mu_\bA(f\circ \pi^\bA_S)$. The function
$f\circ \pi^\bA_S$ is already invariant under $g\in G_\bA$ which are the identity in the $S\cup\{\infty\}$ places, so it is enough to prove this for $g\in G_S$, and then $g(f\circ \pi^\bA_S)=g(f)\circ \pi^\bA_S$. The proof is completed by noting that the measure $\mu_S=(\pi_S^\bA)_*(\mu_\bA)$ satisfies the conditions of Theorem~\ref{thm:Lifting_to_finite_is_Haar} so it is the Haar measure on $X_S$ and hence invariant under $g$.

\end{proof}

\begin{corollary}\label{cor:real_to_adele}
Let $\nu_i\in \cM(X_\bA)$ be $A_\bR$-invariant measures and set 
$\tilde{\nu}_i = (\pi^\bA_\bR)_*(\nu_i)\in \cM(X_\bR)$ which are 
also $A_\bR$-invariant. Then $[\tilde{\nu_i}]\to [\mu_{\bR,Haar}]$ if 
and only if $[\nu_i]\to [\mu_{\bA,Haar}]$.
\end{corollary}
\begin{proof}
The if part is obvious. For the only if part, we first note that the set of $A_\bR$-invariant measures
is a closed subset (both in $X_\bR$ and in $X_\bA$). By Theorem~\ref{thm:lift_Haar_adeles},
the Haar measure $\mu_{\bA,Haar}$ is the unique preimage of $\mu_{\bR,Haar}$ in the set
of $A_\bR$-invariant measures. Thus, since $\pi^\bA_\bR$ is proper, we can apply 
Theorem~\ref{thm:locally_finite_converge} to deduce that $[\nu_i]\to [\mu_{\bA,Haar}]$.
\end{proof}

\subsection{\label{subsec:lifting_orbit}Lifts of orbit measures}

By Theorem~\ref{thm:main}, we know that the averages of the measures $\delta_{p/q}^{\left[0,2\ln q \right]}$
converge to the Haar measure on $X_\bR=X_2=\PGL_2(\bZ)\backslash \PGL_2(\bR)$ as $q\to \infty$. 
In this section we show how to extend these measures to locally finite $A_\bR$-invariant measures on $X_\bR$, and 
relate their averages to projections of single orbit measures in $X_\bA$.

\begin{definition}\label{def:orbit_measure}
Given a homogeneous space $Z = \Ga_0\backslash G_0$, a unimodular group $H<G_0$, and a closed orbit 
$zH$, we denote by $\mu_{zH}$ the \textit{orbit measure}, namely the pushforward of a restriction of a fixed Haar measure
on $H$ to a fundamental domain of $\on{stab}_{H}(z)$ by the orbit map $h\mapsto zh$. The fact that the orbit is closed 
and the unimodularity of $H$ imply that the orbit measure is locally finite and $H$-invariant. Moreover, up to scaling 
this is the unique 
$H$-invariant locally finite measure supported on $zH$.


For an integer $q$, we write $\mu_q := \sum_{p\in \inv{q}} \delta_{p/q}$, $\mu_{p/q A} := \mu_{x_0 u_{p/q} A}$ and 
$\mu_{qA} := \sum_{p\in \inv{q}} \mu_{p/q A}$.
\end{definition}
We note that $\frac{1}{2\ln q }\mu_{p/qA}-\delta_{p/q}^{\left[0,2\ln q \right]}$ is a positive measure
which is supported on the part of the orbit $x_0 u_{p/q}A$
which goes directly to the cusp. Hence, if $f$ is continuous with
compact support, we expect that its integral with respect to this difference will be small. 
This leads us to the following Lemma  which together with Theorem~\ref{thm:main} imply
Theorem~\ref{thm:main2} as a corollary.
\begin{lemma}
\label{lem:to_locally_finite}For any $f\in C_c(X_\bR)$ we have  
$$\limfi q{\infty}\left|\left[\frac{1}{2\ln(q)}\mu_{qA}-\mu_{q}^{\left[0,2\ln q \right]}\right]\left(f\right)\right|=0.$$ 
\end{lemma}
\begin{proof}
Since $f$ is compactly supported, $supp\left(f\right)\subseteq X_2^{\leq M}$ for some $M>0$.
For any $p\in \inv{q}$ we have that $\Gamma u_{p/q}a\left(t\right)=\Gamma\left(\begin{smallmatrix}
e^{-t/2} & \frac{p}{q}e^{t/2}\\
0 & e^{t/2}
\end{smallmatrix}\right)\in X_2^{>M}$ for all $t\notin[-2\ln(M),2\ln(q)+2\ln(M)]$
so that $f$ is zero there, implying that
\begin{align*}
\left|\left[\frac{1}{2\ln(q)}\mu_{qA}-\mu_{q}^{\left[0,2\ln q \right]}\right]\left(f\right)\right| & \leq\frac{1}{2\ln q }\frac{1}{\varphi\left(q\right)}\sum_{\left(p,q\right)=1}\norm f_{\infty}\cdot4\ln\left(M\right)\\
&=\frac{2\ln\left(M\right)}{\ln q }\norm f_{\infty}\overset{q\to\infty}{\longrightarrow}0.
\end{align*}
\end{proof}

\begin{proof}[Proof of Theorem~\ref{thm:main2}.]
 The proof that $\mu_{qA}\to\mu_{Haar}$ follows from Lemma~\ref{lem:to_locally_finite}
above and Theorem~\ref{thm:main}.
\end{proof}


\begin{definition}
We set $G_{\bA,f}=\prod_{p\in \bP}' \PGL_2(\bQ_p)$ and consider it as a subgroup of $G_\bA$. 
Similarly, we let $A_{\bA,f}=A_\bA \cap G_{\bA,f}$.
\end{definition}
We now turn to the proof of Theorem~\ref{thm:main3}. The strategy will be as follows.
Similarly to the real case, if $\tilde{x}_0=\Gamma_\bA\in X_\bA$, then $\tilde{x}_0 A_\bA$ is a closed orbit and therefore $\mu_{\tilde{x}_0A_\bA}$ is a locally finite $A_\bA$-invariant measure and this remains true if we push this measure by elements from $G_{\bA,f}$. Thus, if $g_i\in G_{\bA,f}$ is a sequence satisfying that the projections of $g_i \mu_{\tilde{x}_0 A_\bA}$ to $X_\bR$ are $\mu_{q_i A}$ with $q_i \to \infty$, then we conclude by Corollary \ref{cor:real_to_adele} and Theorem~\ref{thm:main2} that 
$g_i \mu_{\tilde{x}_0 A_\bA} \to \mu_{\bA,Haar}$. 

Since $\mu_{\tilde{x}_0 A_\bA}$ is $A_\bA$-invariant and $(\pi_\bR ^\bA)_*(hg_i \mu_{\tilde{x}_0 A_\bA})=(\pi_\bR ^\bA)_*(g_i \mu_{\tilde{x}_0 A_\bA})$ for any $h\in K:= \prod_{p\in \bP} \PGL_2(\bZ_p)$, we can consider $g_i$ as elements in $K \backslash G_{\bA,f} / A_\bA$. The next lemma shows that modulo these groups, the $g_i$ have a very simple presentation.

\begin{definition}
For $m\in \inv{n}$ let $\bar{u}_{m/n}:=(u_{m/n},u_{m/n},...)\in G_{\bA,f}$. 
\end{definition}
\begin{lemma}\label{lem:infinity_and_beyond}
The group $G_{\bA,f}$ has a decomposition $G_{\bA,f} = KN'A_{\bA,f}$ where 
$K=\prod_{p\in \bP} \PGL_2(\bZ_p)$ and $N'=\{ \bar{u}_{m/n}: m \in \inv{n}\}$.
Moreover, a sequence $g_i= \bar{u}_{m_i /n_i }, (m_i,n_i)=1$ in $G_\bA / A_\bA$ 
diverges to infinity if and only if $n_i \to \infty$.
\end{lemma}
\begin{proof}
By the Iwasawa decomposition, modulo $K$ from the left and $A_{\bA,f}$ from the right,
any element $g\in G_{\bA,f}$ can be expressed as $(g_{p_{1}},g_{p_{2}},...)$ where 
$g_{p}=\left(\begin{smallmatrix}
1 & \frac{m_{p}}{p^{l_{p}}}\\
0 & 1
\end{smallmatrix}\right),\quad\left(m_{p},p^{l_{p}}\right)=1,\;0\leq m_{p}<p^{l_{p}}$ for every $p$, and  $l_{p}=m_{p}=0$ for almost every $p$. 
Let $S$ be the finite set of primes for which $g_{p}\notin \PGL_{2}\left(\ZZ_{p}\right)$
(i.e. $l_{p}\geq1$) and let $n=\prod p^{l_{p}}\in\NN$. Using the
Chinese reminder theorem we can find $m\in\ZZ$ such that $0\leq m<n$,
$m\equiv_{p^{l_{p}}}m_{p}\left(np^{-l_{p}}\right)$ for each $p\in S$
and in particular we get that $\frac{m-m_{p}\left(np^{-l_{p}}\right)}{n}\in\ZZ_{p}$
for all the primes $p$. Setting $h_{p}=\left(\begin{array}{cc}
1 & \frac{m-m_{p}\left(n/p^{l_{p}}\right)}{n}\\
0 & 1
\end{array}\right)\in \PGL_2(\bZ_p)$, we obtain that 
\[
h_{p}g_{p}=
\left(\begin{smallmatrix}
1 & \frac{m-m_{p}\left(n/p^{l_{p}}\right)}{n}\\
0 & 1
\end{smallmatrix}\right)
\left(\begin{smallmatrix}
1 & \frac{m_{p}}{p^{l_{p}}}\\
0 & 1
\end{smallmatrix}\right)
=\left(\begin{smallmatrix}
1\quad & \frac{m_{p}}{p^{l_{p}}}+\frac{m-m_{p}\left(n/p^{l_{p}}\right)}{n}\\
0\quad & 1
\end{smallmatrix}\right)=\left(\begin{smallmatrix}
1 & \frac{m}{n}\\
0 & 1
\end{smallmatrix}\right)
\]
which produces the decomposition $G_{\bA,f} = KN'A_{\bA,f}$.

The second claim follows from the fact that $K$ is compact.
\end{proof}
To prove Theorem~\ref{thm:main3} we are left to show that $(\pi_\bR ^\bA)_*(\bar{u}_{m_i/n_i} \mu_{\tilde{x}_0 A_\bA})=\mu_{qA}$ which is the content of the following claim.

\begin{claim}\label{claim:decomp_orbit}
If $\left(m,n\right)=1$, 
then the map 
\[
(X_{\bA},\bar{u}_{m/n}\mu_{\tilde{x}_0 A_\bA},A_\bA)\overset{\pi_{\RR}^{\bA}}{\longrightarrow}\left(X_{\bR},\mu_{nA},A_\bA\right)
\]
is a factor map.
\end{claim}
\begin{proof}
Since $stab_{A_\bA}(\tilde{x}_0)=A_\bA \cap \PGL_2(\bQ)$ are the diagonal rational matrices, we obtain that its fundamental domain in $A_\bA$ is 
$$A_{\bA}^{0} :=\left\{ \left(\left(\begin{smallmatrix}
e^{-t} & 0\\
0 & 1
\end{smallmatrix}\right),\left(\begin{smallmatrix}
v_{p} & 0\\
0 & 1
\end{smallmatrix}\right),...\right)\in \GL_{2}\left(\bA\right):t\in\RR,\;v_{p}\in\ZZ_{p}^{\times}\right\}\leq A_\bA .$$
It follows that $\mu_{\tilde{x}_0 A_\bA}=\mu_{\tilde{x}_0 A_\bA^0}$ where the map $a\mapsto \tilde{x}_0 a$ where $a\in A_\bA ^0$ is injective and proper. 

Fixing $n$, we define $\psi_{n}:A_\bA ^0\to\inv{n}$
by
\[
\psi_{n}:A_\bA ^0\overset{\Delta_{n}}{\longrightarrow}\prod_{p\mid n}\ZZ_{p}^{\times}\to\prod_{p\mid n}\left(\nicefrac{\ZZ}{p_{i}^{k_{i}}}\right)^{\times}\to\inv{n},
\]
where $\Delta_n$ is the product over $p\mid n$ of the projections
defined by $\left(\begin{smallmatrix}
a & 0\\
0 & 1
\end{smallmatrix}\right)\mapsto a$.

We claim that $\pi_\bR ^\bA( \tilde{x}_0 \psi_n^{-1} (l/m)\bar{u}_{-m/n}) = x_0 u_{l/n} A$ for any $l\in \inv{l}$, namely, the distinct "cosets" are mapped to the distinct $A$-orbits. 

Let $g=(g_\infty, g_{p_1}, g_{p_2},...)\in \psi_n ^{-1}(l/m)$, so that  $g_p=\left(\begin{smallmatrix}
v_{p} & 0\\
0 & 1
\end{smallmatrix}\right)$ and $v_p \equiv_{p^{k_p}} l/m$ where $k_p=\max \{ k : p^k \mid n\}$ for any $p\in \bP$. Since $u_{l/n} \in \Gamma_\bA$, it follows that $\tilde{x}_0 g \bar{u}_{-m/n}=\tilde{x}_0 (u_{l/n},\bar{u}_{l/n}) g \bar{u}_{-m/n}$. For any $p\in \bP$ we have that 
\[
u_{l/n}g_{p}u_{-m/n}=u_{\left(l-m\cdot v_{p}\right)/n}g_{p},
\] 
and since $l-m\cdot v_p \equiv _{p^{k_p}} 0,$ we obtain that $u_{\left(l-m\cdot v_{p}\right)/n},g_{p}\in \PGL_2(\bZ_p)$. By the definition of $\pi^\bA_\bR$, we conclude that $\pi^\bA_\bR(\tilde{x}_0 g \bar{u}_{-m/n})=x_0 u_{l/n} g_\infty$, hence  $\pi_\bR ^\bA( \tilde{x}_0 \psi_n^{-1} (l/m)\bar{u}_{-m/n}) = x_0 u_{l/n} A$. The measure $(\pi_\bR ^\bA)_*(\bar{u}_{m/n} \mu_{\tilde{x}_0 \psi_n^{-1}(l/m)})$ is $A$-invariant and is supported on the orbit of $x_0 u_{l/n} A$ so it must be $\mu_{ l/n A}$.
The proof is now complete by noting that
$$(\pi_\bR ^\bA)_*(\bar{u}_{m/n}\mu_{\tilde{x}_0 A_\bA})=
\sum_{l\in \inv{n} }(\pi_\bR ^\bA)_*(\bar{u}_{m/n} \mu_{\tilde{x}_0 \psi_n^{-1}(l/m)})
\sum_{l\in \inv{n} }\mu_{l/n A}=\mu_{qA}.$$
\end{proof}
Finally, we have all the ingredients to prove Theorem~\ref{thm:main3}.
\begin{proof}[Proof of Theorem~\ref{thm:main3}]
By Lemma~\ref{lem:infinity_and_beyond} we may assume without loss of generality that $g_i = \bar{u}_{m_i/n_i}$ and 
$n_i\to\infty$.
By Theorem~\ref{thm:main2} the measures $[\mu_{n_iA}]$ converge to the homothety class of the 
Haar measure on $X_\bR$ as $i\to\infty$. 
By Claim~\ref{claim:decomp_orbit}  $(\pi^{\bA}_\bR)_*\bar{u}_{m_i/n_i}\mu_{\tilde{x}_0 A_\bA} = \mu_{n_iA}$ 
 so we can apply Corollary~\ref{cor:real_to_adele} to conclude
that $[g_i\mu_{\tilde{x}_0 A_\bA}]$ converge to the homothety class of the Haar measure on $X_\bA$ as desired.
\end{proof}

\section{\label{sec:Application-to-CFE}From the geodesic flow to the Gauss map}
In this section we translate the results obtained in \S\ref{sec:max_entropy_for_partial_orbits} to derive 
consequences on continued fraction expansion (c.f.e). Using a certain cross-section for the flow $a(t)$ on $X_2$
we relate the partial-orbit measures  $\delta_{p/q}^{[0,2\ln(q)]}$ to the normalized counting measures of the finite 
orbit in $[0,1]$ of $p/q$ under the Gauss map.

We begin by recalling the connection between the continued fraction expansion 
and the geodesic flow on the quotient of the hyperbolic plane
$\HH$ by the action of $\PSL_2(\bZ)$ by M\"obius transformations. We keep the exposition brief and refer the 
reader to the
the book of Einsiedler and Ward
\cite[section 9.6]{einsiedler_ergodic_2010} for a detailed account. We bother to repeat many of the things written there
as we are mostly concerned with divergent geodesics which form a null set completely ignored in their discussion.

Identifying the unit tangent bundle $T^{1}\HH$ of the hyperbolic plane with $\PSL_{2}\left(\RR\right)$
we get that every matrix $g=\left(\begin{smallmatrix}
a & b\\
c & d
\end{smallmatrix}\right)\in \PSL_2(\bR)$ defines a unique geodesic in $\HH$ with endpoints 
\begin{align*}
\alpha\left(g\right) & :=\limfi t{\infty}\left(\begin{array}{cc}
a & b\\
c & d
\end{array}\right)\left(\begin{array}{cc}
e^{t/2} & 0\\
0 & e^{-t/2}
\end{array}\right)i
=\limfi t{\infty}\frac{ae^{t}i+b}{ce^{t}i+d}=\frac{a}{c},\\
\omega\left(g\right) & :=\limfi t{\infty}\left(\begin{array}{cc}
a & b\\
c & d
\end{array}\right)\left(\begin{array}{cc}
e^{-t/2} & 0\\
0 & e^{t/2}
\end{array}\right)i
=\limfi t{\infty}\frac{a\frac{i}{e^{t}}+b}{c\frac{i}{e^{t}}+d}=\frac{b}{d}.
\end{align*}
Following Einsiedler and Ward (see Figure~\ref{fig:geodesics}) we define 
\begin{align*}
C_{+} & =\left\{ g\in A\cdot \SO_{2}\left(\RR\right):\alpha(g)\leq -1<0< \omega(g)<1\right\} \\
C_{-} & =\left\{ g\in A\cdot \SO_{2}\left(\RR\right):-1<\omega(g)< 0<1\leq \alpha(g)\right\} \\
C & =C_{+}\cup C_{-},
\end{align*}
considered as subsets of $\PSL_2(\bR)$.

\begin{figure}[h]
\includegraphics[bb=0bp 5bp 669bp 338bp,scale=0.4]{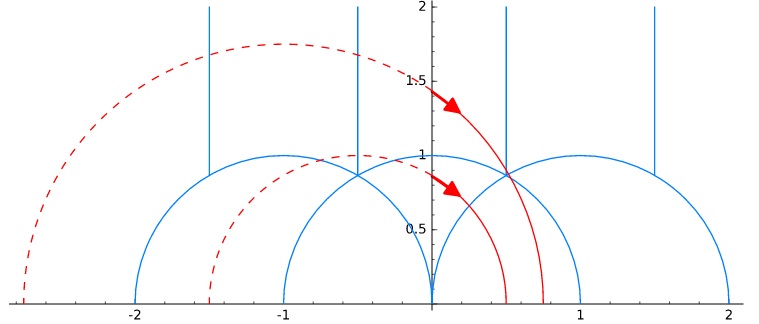}

\caption{The arrows above represent two elements from $C_+$. $C^-$  is obtained by reflection through the $y$-axis of $C_+$.}\label{fig:geodesics}
\end{figure}
We leave the following simple proposition to the reader.
\begin{proposition}
The projection $\pi:\PSL_{2}\left(\RR\right)\to X_{2}=\PSL_{2}(\ZZ)\backslash\PSL_{2}\left(\RR\right)$
restricts to a homeomorphism on $C$.%
\end{proposition}
Henceforth, we will identify $C$ with $\pi(C)$ and denote points there by $g, \bar{g}$ respectively. This will allow us to speak of the start point $\al(\bar{g})$  and end point 
$\om(\bar{g})$ for $\bar{g}\in \pi(C)$. For such $\bar{g}$ we will write $sign(\bar{g})\in\{\pm 1\}$ according to the set
$C_+$ or $C_-$ for which $g$ belongs to.

Our next goal is to show that the Gauss map is a factor of the first return map of the geodesic flow 
on $X_2$ to $\pi\left(C\right)$.
We start by defining a coordinate system on $C$.  Consider the set
$$\tilde{Y}=\left\{ \left(y,z\right):y\in\left(0,1\right),\;0<z\leq \frac{1}{1+y}\right\}\times \{\pm 1\} \subseteq\RR^{2}\times \{\pm 1\}$$
and note that the map from $C$ to $\tilde{Y}$ given by
$$\bar{g}\mapsto \left(\av{\omega(\bar{g})},\frac{1}{\av{\omega(\bar{g})-\alpha(\bar{g})}},sign(\bar{g})\right)$$
is a homeomorphism. In what follows we will always use these coordinates. 
\begin{definition}
Let $\bar{g}\in\pi(C)$. We define the \emph{return
time} $r_{C}(\bar{g})$ and the \emph{first return map} $T_{C}(\bar{g})$
to be
\begin{align*}
r_{C}(\bar{g}) & :=\min\left\{ t>0:\bar{g}\cdot a\left(t\right)\in\pi(C)\right\} \\
T_{C}(\bar{g}) & :=\bar{g}\cdot a\left(r_{C}(\bar{g})\right)\in\pi(C).
\end{align*}
This map is defined only when the forward orbit $\bar{g}\cdot a(t), t>0$ meets $\pi(C)$. Otherwise, we will write $r_C(\bar{g})=\infty$.
\end{definition}
\begin{remark}
While it is not trivial, it is not difficult to show that the minimum
in the definition of $r_{C}(\bar{g})$ is well defined (and not
just the infinimum). Moreover, $r_{C}(\bar{g})$ is uniformly
bounded from below, i.e. ${\displaystyle \inf_{g\in\pi(C)}}r_{C}(\bar{g})>0$.
\end{remark}
We now use the return time map in order to extend our coordinate system.
\begin{lemma}
Let $\hat{Y}=\{(\bar{g},t) : 0<t<r_C(\bar{g})\}\subseteq \tilde{Y}\times \bR$ and set $\theta:(\bar{g},t)\mapsto \bar{g}\cdot a(t)$. If $\mathrm{dm}$ is the restriction of the product measure on $\tilde{Y}\times \bR$ to $\hat{Y}$, then $\kappa \theta_*(\mathrm{dm})=\mu_{Haar}$ for some $\kappa>0$, or equivalently for any $f\in C_c(X_2)$ we have that
\begin{equation}
\int_{X_2}f\left(x\right)\dmu_{Haar}=\kappa\int_{\left(y,z,\epsilon\right)\in Y }\left(\int_{t=0}^{r_{C}\left(y,z,\epsilon\right)}f\left(\left(y,z,\epsilon\right)a\left(t\right)\right)\dt\right)\dmu_{Leb}.\label{eq:orbit_to_haar}
\end{equation}
\end{lemma}
\begin{proof}
This follows from the proof Proposition 9.25 in \cite{einsiedler_ergodic_2010}.
\end{proof}
The connection between the geodesic flow and the Gauss map is given 
in the following two lemmas.
\begin{lemma}[Lemma 9.22 in \cite{einsiedler_ergodic_2010}]
\label{lem:return_map}Under the identification $\pi(C)\simeq \tilde{Y},$
the first return map (where it is defined) is given by
\[
T_{C}\left(y,z,\epsilon\right)=\left(T\left(y\right),y\left(1-yz\right),-\epsilon\right)
\]
where $T(x)=\frac{1}{x}-\flr{\frac{1}{x}}$ is the Gauss map.
\end{lemma}
\begin{lemma}\label{lem:first_meet}
Let $0<x<\frac{1}{2}$ where $x\neq \frac{1}{n}, n\in \bN$. The first time that the orbit $\Gamma u_x a(t),\; t\in \bR$ meets $\pi(C)$ is at the point $(T(x),x,-1)$ for some $t\geq 0$. Similarly for $\frac{1}{2}<x<1, x\neq 1-\frac{1}{n}$, the first meeting is at $(T(1-x),1-x,1)$. 
If $x=\frac{p}{q}$ is rational, then the last time the orbit meets $\pi(C)$ is for some $t\leq 2\ln(q)$.
Finally, we have that $T^2(x)=T(1-x)$ for $\frac{1}{2}<x<1$.
\end{lemma}
\begin{proof}
The proof of the statements involving the first meeting points is essentially the same as the proof of 
Lemma 9.22 in \cite{einsiedler_ergodic_2010} and we leave it to the reader. 
For the statement involving the last meeting time, we note that $\Gamma u_{p/q}a(2\ln(q))=\Gamma u_{p'/q}\left(\begin{smallmatrix}0&-1\\1&0\end{smallmatrix}\right)$, where $pp'\equiv_q 1$ which as a point in $\bH$ is in the standard fundamental domain which points directly up to the cusp, hence its forward orbit doesn't pass through $\pi(C)$.

For the second result, let $0<x<\frac{1}{2}$, so that $x=[0;a_1,a_2,a_3,...]$ with $a_1\geq 2$. We claim that 
$y=[0;1,a_1-1,a_2,a_3,...]$ is equal to $1-x$. Indeed, the c.f.e of $y$ implies that 
$$y=\frac{1}{1+\frac{1}{a_1-1+T(x)}}=\frac{1}{1+\frac{1}{-1+1/x}}=1-x.$$
\end{proof}
The next step is to push measures on $X_2$ to measures on $[0,1]$ and we do it
by lifting functions on $\left[0,1\right]$
to functions on $\SL_{2}\left(\ZZ\right)\backslash\SL_{2}\left(\RR\right)$.
The idea is to define the function first on $\pi(C)$
and to thicken it along the $A$-orbits since $\pi(C)$
has zero measure.
\begin{definition}\label{def:lift_func}
Let $r_*=\frac{1}{2}{\displaystyle \inf_{g\in\pi(C)}}r_{C}(g)>0$. 
For a function $f:\left[0,1\right]\to\RR$ we define $\tilde{f}:X_2 \to\RR$
as follows:
\[
\tilde{f}\left(g\right)=\begin{cases}
\frac{1}{r_*} f(\left|\omega\left(g_{0}\right)\right|) & g=g_{0}a\left(t\right)\;s.t.\;g_{0}\in\pi(C)\;\text{and }0<t< r_*\\
0 & else
\end{cases}
\]
\end{definition}


In general, given a probability measure $\mu$ on $X_2$, we would like to define a measure $\nu$ on $[0,1]$ by setting
$\nu(f):=\mu(\tilde{f})$ for any continuous function $f$. The problem is that $\mu(\tilde{f})$ is not well defined since $\tilde{f}$ is not continuous with compact support. Fortunately, when $\mu=\delta_x^{[0,R]}$ is any partial orbit measure, $\mu(f)$ is well defined and we obtain the following.
\begin{definition}
For a rational $s=\frac{p}{q}\in \bQ$ in reduced form we denote by $\len(p/q)$ the first integer $i$ such that $T^i(p/q)=0$. We define the two measures:
$$ \nu_{p/q}=\frac{1}{\len(p/q)}\sum _{i=0}^{\len(p/q)-1} \delta_{T^i(p/q)}\; ; \;  
\tilde{\nu}_{p/q}=\frac{1}{2\ln(q)}\sum _{i=0}^{\len(p/q)-1} \delta_{T^i(p/q)}$$
\end{definition}
\begin{lemma}\label{lem:pushing_measures}
For any $p\in \inv{q}$ with $q>2$ and $p\ne 1, q-1$, for any $f:[0,1]\to \bR$ we have that $|\delta^{[0,2\ln(q)]}_{p/q}(\tilde{f})-\tilde{\nu}_{p/q}(f)|<\frac{2}{\ln(q)}\norm{f}_\infty$
\end{lemma}
\begin{proof}
Let $t_1<t_2<t_3<\cdots <t_n$ be the times in which the partial orbit $\Gamma u_{p/q} a(t),t\in [0,2\ln(q)]$ meets $\pi(C)$ and set $\bar{g}_i=(y_i ,z_i, \epsilon_i)\in \pi(C)$ to be the corresponding points. It then follows that
$$|\delta_{p/q}^{[0,2\ln(q)]} (\tilde{f}) - \frac{1}{2\ln(q)} \sum_1^{n} f(y_i)|\leq 2\frac {\norm{f}_\infty}{2\ln(q)}.$$

By Lemma~\ref{lem:return_map}, we have that $y_{i+1} = T^i(y_1)$ for all $1\leq i \leq n-1$ and by Lemma~\ref{lem:first_meet} we have that $y_1$ is either $T(\frac{p}{q})$ when $\frac{p}{q}<\frac{1}{2}$ or $T(1-\frac{p}{q})=T^2(\frac{p}{q})$ when $\frac{p}{q}>\frac{1}{2}$, so in any case the $y_i$ are in the $T$-orbit of $\frac{p}{q}$. Finally, Lemma~\ref{lem:first_meet} also tells us that $y_n$ is the last point in the $T$-orbit of $\frac{p}{q}$, so we conclude that 
{\small
$$|\delta^{[0,2\ln(q)]}_{p/q}(\tilde{f})-\tilde{\nu}_{p/q}(f)|=|\delta_{p/q}^{[0,2\ln(q)]} (\tilde{f}) - \frac{1}{2\ln(q)} \sum_0^{\len(p/q)-1} f(T^i\left(\frac{p}{q}\right))|\leq \frac{2}{\ln(q)}\norm{f}_\infty.$$
}
\end{proof}
\begin{remark}
We note that while $\tilde{\nu}_{p/q}$ appear ``naturally", they are not probability measures. Once we show that such a sequence of measures converge to the probability measure $\nuga$, we immediately get that their probability normalization, namely $\nu_{p/q}$, also converge to $\nuga$.
\end{remark}
\begin{lemma}\label{lem:haar_to_gauss}
Let $p_i \in \inv{q_i}$ such that $\delta_{p_i /q_i}^{[0,2\ln(q_i)]}\wstar\mu_{Haar}$. Then $\tilde{\nu}_{p_i/q_i}\wstar 2\ln(2)\kappa \nuga$ and therefore $\frac{\len(p_i/q_i)}{2\ln(q_i)}\to 2\ln(2)\kappa$ and $\nu_{p_i/q_i}\wstar \nuga$.
\end{lemma}
\begin{proof}
Given a segment $I\subseteq [0,1]$ with endpoints $0\leq a<b\leq 1$, we have that $\tilde {\chi}_I = \frac{1}{r} \chi_{\Omega_I}$ where 
\[
\Omega_I=\{ \bar{g}_{0}a(t)\in X:\bar{g}_{0}\in\pi(C)\;,\;0< t< r_*,\;|\omega(g_0)|\in I\}.
\]
The boundary of this set is contained in $F_1\cup F_2\cup F_3 \cup F_4\cup F_5$, where 
\begin{align*}
F_1&= \pi(C),\\
F_2&=\pi(C)a(r_*),\\ 
F_3&= \set{\pi(g)a(t) : g\in A\cdot \SO_2(\bR),  t\in[0,r_*] , |\omega(g)|\in\set{0,1}},\\
F_4&=\set{\pi(g)a(t) : g\in A\cdot \SO_2(\bR),  t\in[0,r_*] , |\alpha(g)|\in1},\\
F_5&=\set{\pi(g)a(t) : g\in A\cdot \SO_2(\bR),  t\in[0,r_*] , |\omega(g)|\in\set{a,b}}.
\end{align*}
In any case this is a null set for $\mu_{Haar}$. 
Since $\del_{p_i /q_i}^{[0,2\ln(q_i)]}\wstar\mu_{Haar}$, for any measurable $B$ with boundary which is $\mu_{Haar}$-null, 
we have $\delta_{p_i /q_i}^{[0,2\ln(q_i)]}(B)\to \mu_{Haar}(B)$ and in particular,
{\small
$$\delta_{p_i /q_i}^{[0,2\ln(q_i)]}(\Omega_I) \to \mu_{Haar}(\Omega_I)=\kappa\int_{(y,z,\epsilon)\in Y}\int_0^{r_C(y,z,\epsilon)} \chi_{\Omega_I} \dmu_{Leb}
=2r_*\kappa\int_a^b \frac{1}{1+s} \mathrm{ds}.$$}
Applying Lemma~\ref{lem:pushing_measures}, we obtain that $\tilde{\nu}_{p_i/q_i}(\chi_I)\to 2\ln(2)\kappa \nuga(\chi_I)$. This result can be extended to any $f\in C[0,1]$ by noting that (1) each such $f$ can be approximated by step function and (2) the measures $\tilde{\nu}_{p/q}$ are uniformly bounded (this follows from the fact that $\len(p/q)\leq 2\log_2(q)$).

Now that we have that $\tilde{\nu}_{p_i/q_i}\wstar 2\ln(2)\kappa \nuga$, evaluating at the constant function 1 produces $\frac{\len(p_i/q_i)}{2\ln(q_i)}\to 2\ln(2)\kappa$ which in turn implies that $\nu_{p_i/q_i}= \frac{2\ln(q_i)}{\len(p_i/q_i)}\tilde{\nu}_{p_i/q_i}\wstar \nuga$.
\end{proof}

\begin{proof}[Proof of Theorem ~\ref{thm:main_app}]
By Corollary~\ref{cor:almost_all_converge}, there exist sets $W_q\subseteq \inv{q}$ with $\displaystyle{\lim_{q\to \infty}} \frac{|W_q|}{\varphi(q)}=1$, such that for any choice  of $p_q \in W_q$ we have that $\delta_{p_q/q}^{[0,2\ln(q)]} \wstar \mu_{Haar}$. Without 
loss of generality we may assume that $1,q-1\notin W_q$ (this assumption is not really necessary as this follows automatically since 
$\del_{1/q}^{[0,2\ln(q)]}, \del_{1/q}^{[0,2\ln(q)]}$ cannot converge to $\mu_{Haar}$). The computation $\kappa=\frac{1}{2\zeta(2)}$ will be done in Theorem~\ref{thm:kappa} below, hence
applying Lemma~\ref{lem:haar_to_gauss} we obtain that $\frac{\len(p_q/q)}{2\ln(q)}\to \frac{\ln(2)}{\zeta(2)}$ and $\nu_{p_q/q}\wstar \nuga$ for such sequences.
\end{proof}

Finally, we compute the value of $\kappa$. One way of doing it is to note that we already know that 
$\frac{1}{\varphi(q)}\sum_{p\in \inv{q}}\frac{\len(p/q)}{2\ln(q)}\to 2\ln(2)\kappa$. This limit was computed 
by Heilbronn in \cite{heilbronn_average_1969} which showed that $\kappa=\frac{3}{\pi^2}=\frac{1}{2\zeta(2)}$.
A direct computation using the return time map is done in the following theorem.
\begin{theorem}\label{thm:kappa}
In Equation~\ref{eq:orbit_to_haar} the constant $\kappa$ is equal to $\frac{3}{\pi^2}=\frac{1}{2\zeta(2)}$.
\end{theorem}
\begin{proof}
In order to find $\kappa$ we compute the return time map and then integrate over $f\equiv 1$.
Given the endpoints $\alpha<-1<0<\omega<1$ of $g$ and writing as before
$y=\epsilon\omega, z=\epsilon\frac{1}{\omega-\alpha}, \epsilon\in\{\pm 1\}$, then $g=\left(\begin{smallmatrix}
1-yz & \epsilon y\\
-\epsilon z & 1
\end{smallmatrix}\right)
\left(\begin{smallmatrix}
e^{t/2} & 0\\
0 & e^{-t/2}
\end{smallmatrix}\right)$ for some $t\in\bR$. In particular, if $g\in C_\pm \subseteq A\cdot \SO_2(\bR)$,
then the rows of $g$ are orthogonal, so that $t=-\ln(\frac{z}{y}(1-yz))/2$. Furthermore, setting $(y',z',\epsilon')=(\frac{1}{y}-\flr{\frac{1}{y}}, y(1-yz),-\epsilon)$, we obtain that 
\[
\left(\begin{array}{cc}
-\epsilon\left\lfloor \frac{1}{y}\right\rfloor  & 1\\
-1 & 0
\end{array}\right)\left(\begin{array}{cc}
1-yz & \epsilon y\\
-\epsilon z & 1
\end{array}\right)\left(\begin{array}{cc}
y & 0\\
0 & \frac{1}{y}
\end{array}\right)=-\epsilon\left(\begin{array}{cc}
1-y'z' & \epsilon'y'\\
-\epsilon'z' & 1
\end{array}\right).
\]
We conclude that $r_C (y,z,\epsilon)=-2\ln(y)-\ln(\frac{z}{y}(1-yz))/2+\ln(\frac{z'}{y'}(1-y'z'))/2$. 
It then follows that 
\[
1=2\kappa\int_0^1 \int_0^{\frac{1}{1+y}}
(-2\ln(y)-\ln(\frac{z}{y}(1-yz))/2+\ln(\frac{z'}{y'}(1-y'z'))/2)\mathrm{dz\cdot dy}.
\]
Since the map $(y,z)\mapsto (y',z')$ is measure preserving, we conclude that 
$1=-4\kappa\int_0^1 \frac{\ln(y)}{1+y} \mathrm{dy}=4\kappa\frac{\pi^2}{12}$, hence $\kappa=\frac{3}{\pi^2}=\frac{1}{2\zeta(2)}$.
\end{proof}


We finish by giving the proof that for a fixed $K$, there are very few rationals $p/q$ with $p\in \inv{q}$ such that the coefficients in their c.f.e are bounded by $K$.
\begin{proof}[Proof of Theorem~\ref{thm:Zaremba}]
Fix some $K>1$ and let 
$$\Lambda_{q,K}=\set{p\in \inv{q}:\text{the entries of the c.f.e of }\frac{p}{q}\text{ are bounded by K} }.$$
We first claim that there is some $M=M(K)>1$ such that $\delta_{p/q}^{[0,2\ln(q)]}$ is supported in $X_2^{\leq M}$
for any $p\in \Lambda_{q,K}$.
We give here an elementary proof but the reader may benefit from reviewing \cite[Section 9.6]{einsiedler_ergodic_2010} and try to establish this claim by herself.
 Let $\frac{p}{q}=[0;a_1,a_2,...,a_n]$ with $a_i\leq K$, and assume that 
$\SL_2(\bR) u_{p/q} a(t) \in X^{>M}$ for some $0\leq t\leq 2\ln(q)$. Let $\bar{0}\neq (m, n)\in \bZ^2$ such that $\norm{(m,-n)u_{p/q} a(t)}_\infty \leq \frac{1}{M}$, or equivalently $|m|\leq \frac{e^{t/2}}{M}$ and $|m\frac{p}{q}-n|\leq \frac{1}{Me^{t/2}}$. Without loss of generality, we may assume that $1\leq m\leq \frac{e^{t/2}}{M}\leq \frac{q}{M}$. Letting $\frac{p_i}{q_i}=[0;a_1,...,a_i]$ be the convergents of $\frac{p}{q}$ we have the recursion
condition $q_{i+1}=q_i a_{i+1} + q_{i-1}\leq (a_{i+1}+1)q_i$. Since $q_n=q$ we obtain that $q_{n-1}\geq \frac{q}{a_n+1}\geq \frac{q}{K+1}$, 
so $M>K+1$ implies that $m<q_{n-1}$.

Choose $k$ such that $q_{k-1}\leq m<q_k\leq q_{n-1} \neq q$. Then by the optimality of convergents 
(proposition 3.3. in \cite{einsiedler_ergodic_2010}), we get that
$|\frac{p}{q}-\frac{p_k}{q_k}|<|\frac{p}{q}-\frac{n}{m}|\leq\frac{1}{Mme^{t/2}}$. 
Furthermore, the convergents satisfy $\frac{1}{2q_{k+1}q_k}<|\frac{p}{q}-\frac{p_k}{q_k}|$ (Exercise 3.1.5
in \cite{einsiedler_ergodic_2010}), and hence
$$\frac{Me^{t/2}}{2}<\frac{q_k q_{k+1}}{m}\leq \frac{(a_{k+1}+1)(a_k+1)^2q^2_{k-1}}{m}
\leq (K+1)^3m\leq (K+1)^3 \frac{e^{t/2}}{M}.$$ 
It follows that $M^2<2(K+1)^3$, and therefore the support of $\delta_{p/q}^{[0,2\ln(q)]}$ 
must be contained in $X^{\leq 2(K+1)^2}$.

By the claim that we just proved, the probability measures $\delta_{\Lambda_{q,K}}^{[0,2\ln(q)]}$ 
are all supported in the compact set $X^{\leq 2(K+1)^2}$ so in particular they 
do not exhibit escape of mass. If we also knew that $\frac {\ln|\Lambda_{q,K}|}{\ln(q)} \to 1$,
then applying Theorem~\ref{thm:mainug}, we conclude that $\delta_{\Lambda_{q,K}}^{[0,2\ln(q)]}$ 
converges to the Haar probability measure, but the limit must also be supported on $X^{\leq 2(K+1)^2}$
- contradiction. It follows that $\limsup \frac {\ln|\Lambda_{q,K}|}{\ln(q)}<1$ or equivalently
$|\Lambda_{q,K}|=o(q^{1-\varepsilon})$ for some $\varepsilon>0$.

\end{proof}

\appendix

\section{\label{app:Working-with-balls}the proof of Lemma~\ref{lem:bowen_control}}
Before we give the proof, we need some results about hyperbolic balls. Recall from Definition~\ref{def:balls} that
for $H\leq SL_{2}\left(\RR\right)$, we define the $H$-balls $B_{r}^{H}=\left\{ I+W\in H:\norm W_{\infty}<r\right\} $.
In particular we have
\begin{align*}
B_{r}^{U^{+}}&=\left\{ I+\alpha E_{1,2}:\left|\alpha\right|<r\right\} \\
B_{r}^{U^{-}A}&=\left\{ I+W\in SL_{2}\left(\RR\right):W_{1,2}=0,\;\left|W_{i,j}\right|<r\right\}.
\end{align*} 
We further write $B_{\eta,N}=B_{\eta e^{-N}} ^{U^+} B_\eta ^{U^- A}, B_\eta:=B_{\eta,0}$ and $a=\smallmat{e^{-1/2}&0\\0&e^{1/2}}$ 
(so that $a B_r^{U^+} a^{-1} = B_{r/e}^{U^+}$).
\begin{lemma}
\label{lem:BALLS}Let $H\leq G$ be any subgroup. We have the following:

\begin{enumerate}
\item $\left(B_{K}^{H}\right)^{-1}=B_{K}^{H}$.
\item $B_{K_{1}}^{H}B_{K_{2}}^{H}\subseteq B_{2\left(K_{1}+K_{2}\right)}^{H}$
whenever $K_{1},K_{2}<1$.
\item Suppose that $r^{+},r^{-}<\frac{1}{4}$. Then $B_{r^{-}}^{U^{-}A}B_{r^{+}}^{U^{+}}\subseteq B_{2r^{+}}^{U^{+}}B_{2r^{-}}^{U^{-}A}$.
\item Suppose that $r^{+},r^{-}<\frac{1}{4}$. Then $gB_{r^{+}}^{U^{+}}g^{-1}\in B_{2r^{+}}^{U^{+}}B_{6r^{-}}^{U^{-}A}$
for every $g\in B_{r^{-}}^{U^{-}A}$.
\item Suppose that $r^{+},r^{-}<\frac{1}{16}$ and $x,y\in\Gamma\backslash G$.
Then 
\[
y\in xB_{r^{+}}^{U^{+}}B_{r^{-}}^{U^{-}A}\quad\Rightarrow\quad xB_{r^{+}}^{U^{+}}B_{r^{-}}^{U^{-}A}\subseteq yB_{8r^{+}}^{U^{+}}B_{6r^{-}}^{U^{-}A}.
\]
\end{enumerate}
\end{lemma}
\begin{proof}
\begin{enumerate}
\item Follows from the fact that $\left(\begin{array}{cc}
a & b\\
c & d
\end{array}\right)^{-1}=\left(\begin{array}{cc}
d & -b\\
-c & a
\end{array}\right)$ for matrices of determinant 1.
\item Follows from the identity $\left(I+W_1\right)\left(I+W_2\right)=I+(W_1+W_2)+W_1 W_2$
and the fact that $\norm {W_1 W_2}_\infty \leq 2\norm {W_1}_\infty \norm{W_2}_\infty$.
\item Suppose that $\left|u\right|,\left|v\right|,\left|w\right|<r^{-}$
and $\left|x\right|<r^{+}$. Then
{\tiny
\begin{align*}
\left(\begin{array}{cc}
1+u & 0\\
v & 1+w
\end{array}\right)\left(\begin{array}{cc}
1 & x\\
0 & 1
\end{array}\right)
& =\left(\begin{array}{cc}
1 & \frac{x\left(1+u\right)}{1+w+vx}\\
0 & 1
\end{array}\right)\left(\begin{array}{cc}
1+u-\frac{x\left(1+u\right)}{1+w+vx}v & 0\\
v & 1+w+vx
\end{array}\right)
\end{align*}
}
which is in $B_{2r^{+}}^{U^{+}}B_{2r^{-}}^{U^{-}A}$.
\item Using the previous parts we get that 
\[
gB_{r^{+}}^{U^{+}}g^{-1}\subseteq B_{r^{-}}^{U^{-}A}B_{r^{+}}^{U^{+}}B_{r^{-}}^{U^{-}A}\subseteq B_{2r^{+}}^{U^{+}}B_{2r^{-}}^{U^{-}A}B_{r^{-}}^{U^{-}A}\subseteq B_{2r^{+}}^{U^{+}}B_{6r^{-}}^{U^{-}A}.
\]
\item Using the previous parts, $y=xh^{+}h^{-}$ with $h^{+}\in B_{r^{+}}^{U^{+}}$
and $h^{-}\in B_{r^{-}}^{U^{-}A}$, we have that 
\begin{align*}
xB_{r^{+}}^{U^{+}}B_{r^{-}}^{U^{-}A} & =  y\left(h^{-}\right)^{-1}\left(h^{+}\right)^{-1}B_{r^{+}}^{U^{+}}B_{r^{-}}^{U^{-}A} \subseteq yB_{r^{-}}^{U^{-}A}B_{4r^{+}}^{U^{+}}B_{r^{-}}^{U^{-}A} \\
&\subseteq yB_{8r^{+}}^{U^{+}}B_{2r^{-}}^{U^{-}A}B_{r^{-}}^{U^{-}A}\subseteq yB_{8r^{+}}^{U^{+}}B_{6r^{-}}^{U^{-}A}
\end{align*}
\end{enumerate}
\end{proof}
\begin{lemma}
\label{lem:Ball_cover} There is some constant $C$ such that for
all $0<r_{1},r_{2}$ small enough, $x\in X_2$ and $Y\subseteq xB_{r_{1}}^{U^{+}}B_{r_{2}}^{U^{-}A}$,
there are $y_{1},...,y_{C}\in Y$ such that $Y\subseteq\bigcup y_{i}B_{r_{1}/e}^{U^{+}}B_{r_{2}/e}^{U^{-}A}$.
\end{lemma}
\begin{proof}
We first prove a similar claim in $\SL_{2}\left(\RR\right)$, that
there exists a constant $C_{1}$ such that for all $0<r_{1},r_{2}$
small enough and $R\geq1$ we can find $x_{1},...,x_{C_{1}R^{3}}\in\SL_2(\bR)$
such that $B_{r_{1}}^{U^{+}}B_{r_{2}}^{U^{-}A}\subseteq\bigcup x_{i}B_{r_{1}/R}^{U^{+}}B_{r_{2}/R}^{U^{-}A}$.
Since $U^{+}\cong\RR$, given $R'\geq1$ we can find $O\left(R'\right)$
elements $g_{i}\in B_{r_{1}}^{U^{+}}$ such that $B_{r_{1}}^{U^{+}}\subseteq\bigcup g_{i}B_{r_{1}/R'}^{U^{+}}$,
and similarly we can find $O\left((R')^{2}\right)$ elements $h_{j}\in B_{2r_{2}}^{U^{-}A}$
such that $B_{r_{2}}^{U^{-}}\subseteq\bigcup h_{j}B_{r_{2}/R'}^{U^{-}A}$.
Applying Lemma~\ref{lem:BALLS} we obtain that 
\begin{align*}
B_{r_{1}}^{U^{+}}B_{r_{2}}^{U^{-}} & \subseteq\bigcup_{i,j}g_{i}B_{r_{1}/R'}^{U^{+}}h_{j}B_{r_{2}/R'}^{U^{-}A}=\bigcup_{i,j}g_{i}h_{j}\left(h_{j}^{-1}B_{r_{1}/R'}^{U^{+}}h_{j}\right)B_{r_{2}/R'}^{U^{-}A}\\
 & \subseteq\bigcup_{i,j}g_{i}h_{j}B_{2r_{1}/R'}^{U^{+}}B_{6r_{2}/R'}^{U^{-}A}B_{r_{2}/R'}^{U^{-}A}\subseteq\bigcup_{i,j}g_{i}h_{j}B_{2r_{1}/R'}^{U^{+}}B_{14r_{2}/R'}^{U^{-}A}
\end{align*}
Choosing $R'=14R$ finishes the claim.

We now transfer this result to $X_2$. Let $r_{1},r_{2}>0$
small enough, $x\in X_2$ and $Y\subseteq xB_{r_{1}}^{U^{+}}B_{r_{2}}^{U^{-}A}$.
Setting $R=8e$, we can find $O\left(R^{3}\right)=O\left(1\right)$
many $x_{i}\in \SL_{2}\left(\RR\right)$ such that $xB_{r_{1}}^{U^{+}}B_{r_{2}}^{U^{-}}\subseteq\bigcup x\cdot x_{i}B_{r_{1}/R}^{U^{+}}B_{r_{2}/R}^{U^{-}A}$.
Choose $y_{i}$ such that $y_{i}\in Y\cap x\cdot x_{i}B_{r_{1}/R}^{U^{+}}B_{r_{2}/R}^{U^{-}A}$
if this set is not empty and otherwise choose some $y_{i}\in Y$ arbitrarily.
Since $y_{i}\in x\cdot x_{i}B_{r_{1}/R}^{U^{+}}B_{r_{2}/R}^{U^{-}A}$,
applying Lemma~\ref{lem:BALLS} (5) we get that 
\[
x\cdot x_{i}B_{r_{1}/R}^{U^{+}}B_{r_{2}/R}^{U^{-}A}\subseteq y_{i}B_{8r_{1}/R}^{U^{+}}B_{6r_{2}/R}^{U^{-}A}=y_{i}B_{r_{1}/e}^{U^{+}}B_{r_{2}/e}^{U^{-}A}
\]
which completes the proof.
\end{proof}

\begin{proof}[Proof of Lemma~\ref{lem:bowen_control}]
Choose $\eta_{0}\left(M\right)>0$ to be small enough so that Lemmas~\ref{lem:BALLS} and \ref{lem:Ball_cover}
will be applicable and that the map $g\mapsto xg$ from $B_{\eta}\to\Gamma\backslash G$
is injective for all $x\in X^{\leq M}$. Let $\pp=\left\{ P_{0},...,P_{n}\right\} $ be an
$\left(M,\eta\right)$ partition.

Consider the function $f\left(x\right)=\frac{1}{N}\sum_{0}^{N-1}1_{X^{>M}}\left(T^{i}x\right)$
and note that this function is constant on each $P\in\pp_{N}$.

Setting $X'=X^{\leq M}\cap\left\{ x:f\left(x\right)\leq\kappa\right\} $,
we obtain that
{\small
\begin{align*}
1 & \leq\mu\left(X^{>M}\right)+\mu\left(\left\{ f\left(x\right)>\kappa\right\} \right)+\mu\left(X'\right)\leq\mu\left(X^{>M}\right)+\kappa^{-1}\int f\left(x\right)\mathrm{d\mu}+\mu\left(X'\right)\\
&= \mu\left(X^{>M}\right)+\kappa^{-1}\mu^{N}\left(X^{>M}\right)+\mu\left(X'\right),
\end{align*}
}
thus proving part $(3)$ in the theorem.

For $S\in\pp_{N}$, $S\subseteq X'$ set $V_{m}=\left|\left\{ 0\leq i\leq m\;\mid\;T^{i}\left(S\right)\subseteq X^{>M}\right\} \right|$.
Let $C$ be the constant from Lemma~\ref{lem:Ball_cover}.
We claim that $S\subseteq \bigcup_1^{C^{|V_m|}} y_i B_{\eta,N}$ with $y_i\in S$ for any
$0\leq m\leq N$, and the lemma will follow by setting $m=N-1$.
For $m=0$, let $y\in S\subseteq P_{i}\subseteq x_{i}B_{\frac{\eta}{10}}$
for some $i\geq1$, so by Lemma~\ref{lem:BALLS} $S\subseteq yB_{\eta}$,
thus proving the case for $m=0$.

Assume that $S\subseteq \bigcup_1^{C^{|V_m|}} y_i B_{\eta,m}$ with $y_i\in S$ for $m<N-1$ and we prove for $m+1$.
\begin{itemize}
\item Suppose first that $T^{m+1}S\subseteq X^{\leq M}$ so that $T^{m+1}S\subseteq P_{j}\subseteq x_{j}B_{\frac{\eta}{10}}$
for some $j\geq1$. This case will be complete if 
$S\cap y_{i}B_{\eta,m}=S\cap y_{i}B_{\eta,m+1}$ for every $i$. Indeed,
Lemma~\ref{lem:BALLS} implies that $T^{m+1}S\subseteq x_{j}B_{\frac{\eta}{10}}\subseteq y_{i}a^{\left(m+1\right)}B_{\eta}$, so if $y_{i}g\in S$ with $g\in B_{\eta,m}$, then
\[
\left[y_{i}a^{\left(m+1\right)}\right]a^{-\left(m+1\right)}ga^{\left(m+1\right)}=y_{i}ga^{\left(m+1\right)}\in T^{m+1}S'\subseteq y_{i}a^{\left(m+1\right)}B_{\eta}.
\]
By the assumption on the injectivity radius, we conclude that $g\in B_{\eta,m}\cap a^{\left(m+1\right)}B_{\eta}a^{-\left(m+1\right)}=B_{\eta,m+1}$
which is what we wanted to show.
\item Suppose now that $T^{n+1}S\subseteq X^{>M}$. By Lemma~\ref{lem:Ball_cover},
for each $i$ we have that $S\cap y_{i}B_{\eta,m}\subseteq\bigcup_{j=1}^{C}\tilde{y}_{i}^{(j)}B_{\frac{\eta}{e},m}\subseteq\bigcup_{j=1}^{C}\tilde{y}_{i}^{(j)}B_{\frac{\eta}{e},m+1}$ with $\tilde{y}_i^j \in S$, which completes this case and the proof.
\end{itemize}
\end{proof}

\begin{remark}\label{rem:correction} In the original proof of Lemma 4.5 from \cite{einsiedler_distribution_2012}, there was a slight inaccuracy in the final argument where the center of the balls $yB_{\eta,m}$ were not shown to be inside $S$. This inaccuracy is resolved in Lemma~\ref{lem:Ball_cover}.
\end{remark}

\bibliographystyle{plain}
\bibliography{bib}

\end{document}